\documentclass[preprint,10pt]{article}
\usepackage{amssymb,amsfonts,amsmath,amsthm,amscd,bbold,dsfont,mathrsfs}
\usepackage[pdftex]{graphicx,color}
\usepackage{epsfig}
\usepackage{verbatim}

\DeclareMathAlphabet{\mathpzc}{OT1}{pzc}{m}{it}

\newtheorem{propo}{Proposition}[section]
\newtheorem{lemma}[propo]{Lemma}

\newtheorem{thm}[propo]{Theorem}

\newtheorem{remark}[propo]{Remark}

\def\endproof{\hfill$\Box$\vspace{0.4cm}}

\newcommand{\eqnsection}{\renewcommand{\theequation}{\thesection.\arabic{equation}}
      \makeatletter \csname @addtoreset\endcsname{equation}{section}\makeatother}

\newcommand{\proc}[1]{ #1 (\cdot) }

\def\eps{\epsilon}

\def\th{\theta}
\def\ka{\kappa}
\def\al{\alpha}
\def\la{\lambda}
\def\si{\sigma}
\def\de{\delta}
\def\ta{\tau}
\def\tn{\tau/n}
\def\D{\Delta}
\def\ze{\zeta}
\def\gam{\gamma}
\def\om{\omega}

\def\lta{\underline{\tau}_n}
\def\uta{\overline{\tau}_n}
\def\ltai{\underline{\tau}_n^i}
\def\utai{\overline{\tau}_n^i}

\def\cA{{\cal A}}
\def\cK{{\cal K}}
\def\cC{{\cal C}}
\def\cG{{\cal G}}
\def\cF{{\cal F}}
\def\cD{{\cal D}}
\def\cM{{\cal M}}

\def\cH{{\cal H}}
\def\cV{{\cal V}}

\def\l|{\left|\left|}
\def\r|{\right|\right|}
\def\ll{\left( }
\def\rr{\right) }
\def\lbr{\left\{ }
\def\rbr{\right\} }
\def\ln{\left| }
\def\rn{\right| }
\def\lsq{\left[ }
\def\rsq{\right] }
\def\fl{ \lfloor }
\def\fr{ \rfloor }
\newcommand{\ceil}[1]{ \lceil #1 \rceil }
\newcommand{\floor}[1]{ \lfloor #1 \rfloor }

\def\RR{\mathds R}
\def\P{\mathds P}

\def\E{\mathds E}
\def\1{\mathds 1}

\def\ind{{\mathds I}}

\newcommand{\reals}{{\mathds R}}
\newcommand{\ints}{{\mathds Z}}

\def\A{\mathds A}
\def\B{\mathds B}
\def\tA{\tilde{ \mathds A } }
\def\tB{\tilde{ \mathds B } }

\def\zn{ \vec{z}\,^{(n)} }
\def\En{{\mathds E}_n}
\def\Enr{ {\mathds E}_{n,\rho} }
\def\Pn{{\mathds P}_n}
\def\Pnr{ {\mathds P}_{n,\rho} }
\def\Paux{ \Pnr^{ \mbox{\tiny\rm aux} } }

\def\G{\mathds G}
\def\U{\mathds U}
\def\Q{\mathds Q}

\def\V{\mathds V}

\def\sfA{{\sf A}}
\def\sfD{{\sf D}}
\def\sfN{{\sf N}}

\def\sfk{{\sf k}}

\def\sfd{{\sf d}}
\def\sfe{{\sf e}}
\def\sfg{{\sf g}}
\def\sfh{{\sf h}}
\def\sfR{{\sf R}}
\def\sfG{{\sf G}}
\def\sfS{{\sf S}}

\def\sfm{\vec{\sf m}^\dag}
\def\sfmr{\vec{\sf m}}
\def\sfnr{\vec{\sf n}}

\def\sTV{{\mbox{\tiny\rm TV}} } 
\def\Fr{ {\mbox{\tiny\rm FR}} } 
\def\Tr{ {\mbox{\rm Tr}} }
\def\ex{{\mbox{\tiny\rm exit}} } 
\def\nex{{\mbox{\tiny\rm no exit}} } 
\def\nocore{{\mbox{\tiny\rm no core}} } 

\def\sfWD{ W_n }

\def\sfWR{ \widehat{W}_n }
\def\hW{ \widehat{W} }

\def\maly{ {\varepsilon} }
\def\mala{ {\de} }

\def\zrn{ \z_{n, \rho}  }

\def\sfK{{ \sf K}}

\def\z{\vec{z}}
\def\ztau{\vec{z} (\tau)}
\def\zp{\vec{z} \,'}

\def\Dz{\Delta \vec{z}}

\def\F{\vec{F}}
\def\R{\vec{R}}

\def\x{\vec{x}}
\def\xp{\vec{x} \,'}

\def\y{\vec{y}}

\def\ytn{\vec{y}(\tau/n)}
\def\ysnrho{ \vec{y}(\si/n,\rho)  }
\def\ytnrho{ \vec{y}(\tau/n,\rho) }

\def\ysn{\vec{y}(\si/n)}

\def\thc{ {\theta_c} }
\def\vxi{ \vec{\xi} }

\def\ovth{ { \overline{\th} } }
\def\ovn{  { \overline{n}   } }

\def\calF{{\cal F}}

\def\pip{{\partial}}
\def\did{ \mathrm{d} }

\def\Qst{ { \mathds Q }^\ast }
\def\yst{ \vec{y}\,^\ast }
\def\Fi{ F_i }
\def\Gi{ G_i }
\def\thic{ \theta^i_c }
\def\thjc{ \theta^j_c }
\def\rhoc{ \rho_c }

\def\tast{ \tau_\ast }
\def\bKn{ \mathbf{K}_n }

\def\Var{ \mathrm{Var} }
\def\Cov{ \mathrm{Cov} }

\def\Ai{ \mathrm{Ai} }
\def\Bi{ \mathrm{Bi} }

\def\bYn{ \mathbf{Y}_n }
\def\bZn{ \mathbf{Z}_n }

\def\bHn{ \mathbf{H}_n }
\def\bXn{ \mathbf{X}_n }
\def\bAn{ \mathbf{A}_n }
\def\bGn{ \mathbf{\Gamma}_n }
\def\bG{ \mathbf{\Gamma} }
\def\bSn{ \mathbf{\Sigma}_n }
\def\bS{ \mathbf{\Sigma} }
\def\bVn{ \mathbf{V}_n}
\def\bN{ \mathbf{N} }

\def\bY{ \mathbf{Y} }
\def\bzer{ \mathbf{0} }
\def\aa{ \mathbf{a} }
\def\bb{ \mathbf{b} }
\def\xx{ \mathbf{x} }
\def\yy{ \mathbf{y} }
\def\zz{ \mathbf{z} }

\def\BB{ \mathbf{B} }
\def\AA{ \mathbf{A} }
\def\TT{ \mathbf{T} }
\def\SS{ \mathbf{S} }

\def\bO{ \mathbf{O} }
\def\bv{ \mathbf{v} }
\def\bV{ \mathbf{V} }
\def\bg{ \mathbf{g} }

\def\logp{ \log_+ }
\def\RN{ \reals^N }
\def\vze{ \vec{\zeta} }

\def\bLam{ \mathbf{\Lambda} }

\def\oz{ \z }
\def\ozp{ \zp }
\def\t{ \vec{\tau} }
\def\tp{ \vec{\tau} \,'}
\def\vo{ \vec{\om} }
\def\vop{ \vec{\om} \,'}
\def\orn{ \vo_{n, \rho}  }

\def\amb{ {\cal A} }
\def\dom{ {\sf H} }
\def\domt{ \widetilde{\dom} }
\def\Wt{ \widetilde{W} }

\def\sfx{ {\sf x} }
\def\sfy{ {\sf y} }
\def\sfz{ {\sf z} }
\def\sfv{ {\sf v} }
\def\sfu{ {\sf u}^\rho }

\def\coeff{ {\sf coeff} }
\def\dd{ d }
\def\ovd{ \overline{d} }
\def\ovell{ \overline{\ell} }

\def\hta{ \hat{\tau} }

\def\pp{ {\mathfrak p} }
\def\ss{ {\mathfrak s} }

\def\Poisson{{\sf Poisson}}
\def\dist{{\sf dist}}

\def\Fx{ \vec{F}\,^{(u)} } 
\def\Fy{ \vec{F}\,^{(v)} }

\def\sfV{ {\sf V} }

\def\cn{ \vec{u} }
\def\vn{ \vec{v} }
\def\Qc{ \hat{\Q} }

\def\urj{ u_{\rho,j} }
\def\cGt{\widetilde{\cG}}

\eqnsection

\begin{document}
\title{Sharp approximation for density dependent Markov chains}

\author{Kamil Szczegot\thanks{Department of Mathematics, Stanford University  
\newline\indent
\newline Research partially supported by NSF grant \#DMS-0806211.
\newline
{\bf AMS (2000) Subject Classification:}
{Primary: 60J10, 60F17; Secondary: 82B26, 68W20, 94A29 }
\newline
{\bf Keywords:} Markov chains, limit theorem, exit problem, finite size scaling, core, random hypergraph, low-density parity-check codes} }

\date{\today}

\maketitle 

\begin{abstract}
Consider a sequence (indexed by $n$) of Markov chains $\z\,^{(n)}(\cdot)$ in  $\RR^d$ characterized by transition kernels that approximately (in $n$) depend only on the rescaled state  $n^{-1} \z\,^{(n)}(\ta)$. Subject to a smoothness condition, such a family can be closely coupled on short time intervals to a Brownian motion with quadratic drift. This construction is used to determine the first two terms in the asymptotic (in $n$) expansion of the probability that the rescaled chain exits a convex polytope. The constant term and the first correction of size $\Theta(n^{-1/6})$ admit sharp characterization by solutions to associated differential equations and an absolute constant. The error is smaller than $O(n^{-\eta})$ for any $\eta < 1/4$.

These results are directly applied to the analysis of randomized algorithms at phase transitions. In particular, the `peeling' algorithm in large random hypergraphs, or equivalently the iterative decoding scheme for low-density parity-check codes over the binary erasure channel is studied to determine the finite size scaling behavior for irregular hypergraph ensembles.
\end{abstract}

\section{Introduction}

We deal with a sequence of Markov chains $\proc{\zn}$ taking values in a countable set $\sfS \subset\RR^d$, started at an initial distribution $\z(0)$, and with transition probabilities given by a sequence of kernels
\begin{eqnarray}\label{def:Wn}
W_n(\D | \zn(\ta)) \equiv \P ( \zn(\tau +1) - \zn(\ta) = \D | \zn(\ta))
\end{eqnarray}
satisfying $W_n(\D|\z) \approx W(\D|n^{-1}\z)$, where $W$ has compact support in the first coordinate and is smooth in the second coordinate. Markov chains in this class are in a sense `slowly-varying': they make jumps of size $\Theta(1)$, but the kernel depends only on transitions on the scale of $\Theta(n)$.  The study of pure-jump Markov processes with an analogous property (transition intensities of the form $n \beta(n^{-1}\cdot)$) goes back at least to \cite{Ku70}. They have since been widely applied to population processes in epidemics, queueing, or networks (cf. \cite{Ku81}, \cite{AB00}, \cite{Wh02} and references therein), as well as models for chemical reactions (cf. \cite{BKPR06} for a recent contribution). These applications motivate the name \emph{density dependent}, since $n^{-1}\zn$ can be thought of as the density of a population of $\zn$ individuals living in an area of size $n$. Another application, particularly relevant here, is the analysis of randomized algorithms arising in (probabilistic) combinatorics and optimization (cf. \cite{Wor95} and others).

Letting $\z(0) = n \y_0$ for a non-random (for now) $\y_0$, it is known that the rescaled, interpolated process $n^{-1}\zn(n\,\cdot)$ on compact time intervals $[0,\ovth]$ concentrates near (its \emph{fluid limit}) the solution to the ordinary differential equation
\begin{eqnarray}\label{eq:ODE}
\frac{\did \y}{\did \th}(\th) = \F(\y(\th))\,,\qquad \y(0) = \y_0,
\end{eqnarray}
where $\F (\x) \equiv \sum_{\D}\D \; W(\D | \x)$.  It is a consequence of the results on convergence of Markov semigroups that the scaled fluctuation of the process about this solution, $\sqrt{n}\lbr n^{-1}\z(n \,\cdot) - y(\cdot)\rbr$, converges weakly to a diffusion. Both results are special cases in the rich literature on convergence of Markov processes: cf. \cite{EK86, Kal02} for general theory, \cite{Wh02} as well as work of M.~Bramson, M.~Harrison, and R.~Williams for fluid and diffusive limits in queueing theory, and \cite{DN08} for a recent survey of fluid limits with quantifiable error probabilities (in the spirit of Lemma \ref{prop:ode} here).

In view of these results, one can think of the rescaled process as a random perturbation of the vector flow $\F$ and the question arises of how randomness changes the dynamics. Systems with Gaussian or Markov perturbation are a central theme of \cite{FW84} (where among other results large deviation estimates for the place of exit from a domain $\sfD$ are derived) and research following from there, of which the most relevant to the problem to be defined here is the case of characteristic boundaries (cf. \cite{Day89} and others). Asymptotic expansions of probabilities for chains with two time-scales, i.e. a random perturbation of a Markov chain have been surveyed in \cite{YZ05}. There are also extensions of fluid limits with error estimates, cf. \cite{Tur07} for a recent contribution.
In contradistinction to these results, we are interested in the fine asymptotics of the probability that the rescaled Markov chain exits a domain $\sfD$ before some fixed time $\ovth > 0$ and when this probability does not tend to either 0 or 1, which happens when the vector field $\F$ is tangent to the boundary. This is apparent from the left panel of Figure \ref{fig:A}, where the bound coming from the fluid limit estimate does not prevent the chain from exiting the domain, as it would should the deterministic evolution $\y$ remain in the interior of $\sfD$. 
\begin{figure}[t]
\begin{center}
\resizebox{3.5in}{!}{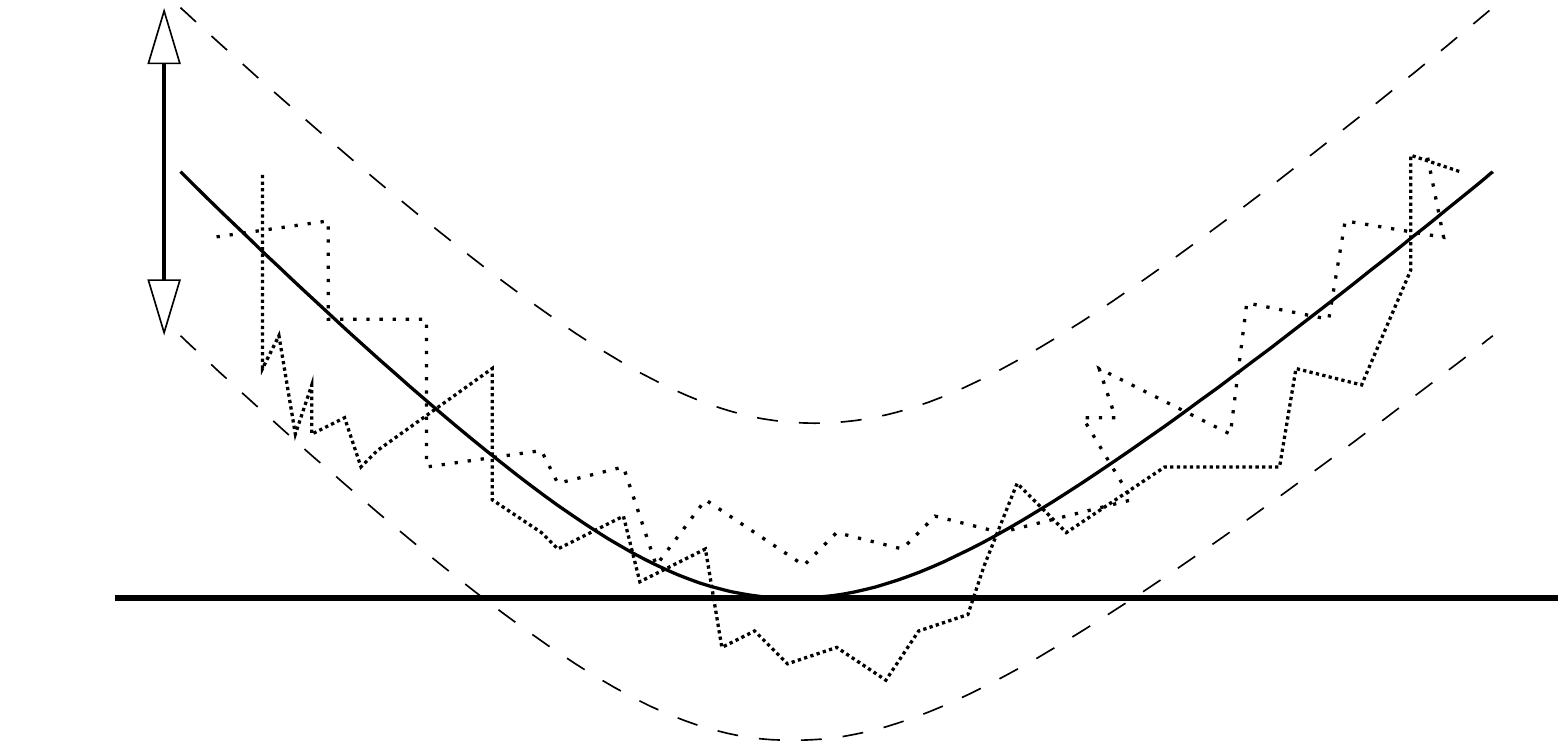}
\qquad
\resizebox{2.7in}{!}{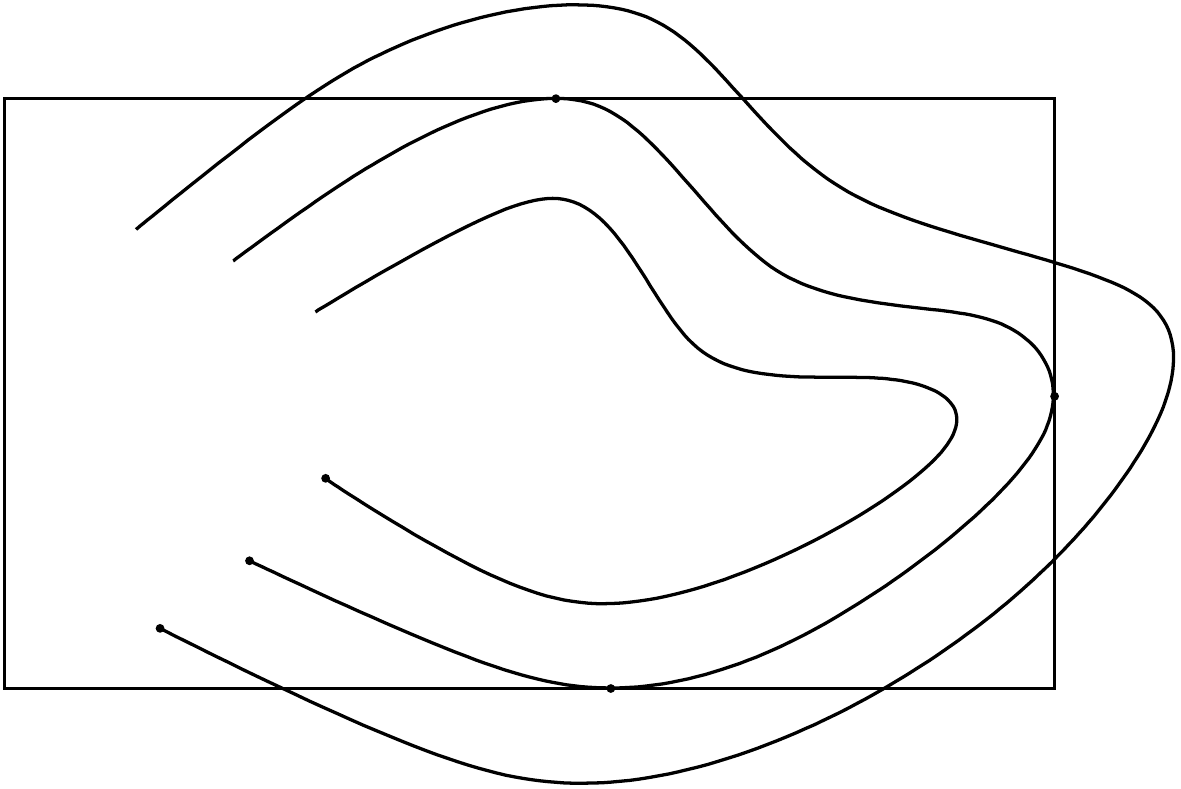}
\caption{{\bf Left panel:} The interesting case of the vector field $\F$ tangent to the boundary $\pip \sfD$. Two realizations of the chain $n^{-1}\zn(n\,\cdot)$ (dotted lines) remain within distance $O(\sqrt{(\log n)/n})$ of the deterministic evolution $\y$ (solid parabola); one realization exits the domain, while the other remains in the interior. {\bf Right panel:}
Example setting in $\RR^2$. Here $\sfD$ is a rectangle and the critical trajectory started at $y_0$ touches the boundary of $\sfD$ at $N=3$ critical times $\theta_c^1$, $\theta_c^2$, and $\theta_c^3$}\label{fig:A}
\end{center}
\end{figure}

To investigate this probability as a function of the initial condition, we introduce the following construction. 
Let $\rho \mapsto y_\rho$ be a smooth parametrization (from $(-\de,\de)$ into $\sfD$) of the initial condition, such that the solution of (\ref{eq:ODE}) with $y(0) = y_0$ remains in the interior $\sfD^\circ$ of $\sfD$ except at finitely many \emph{critical times} $\thic$, $i \leq N$, when the trajectory is tangent to the boundary. (See Figure \ref{fig:A} for an example and Section \ref{sec:introexit} for a rigorous statement.)  It follows from the error bounds on the fluid limit available in our setting, that the probability $P_\ex(n, \rho)$ of the exit from $\sfD$ by chain $n^{-1} \zn(\cdot)$ started at $n^{-1}\zn(0) = \y_\rho$ before time $\ovn \equiv \floor{n\ovth}$ equals 0 or 1 (up to $O(e^{-c n})$), depending on whether the trajectory  $\y(\cdot)$ started at $\y_\rho$ remains in $\sfD^\circ$ or exits $\sfD$, respectively. (See also \cite[Section 4.3]{DN08} for a basic result on the place of exit.) 
Thus we have a phase transition for parameter $\rho$ and as a consequence of the diffusive limit result (assuming the chain is non-degenerate) the size of the scaling window of this transition is $O(n^{-1/2})$, meaning that $P_\ex(n, rn^{-1/2}) = f_1(r) + o(1)$ as $n \rightarrow \infty$ for some smooth function $f_1$. The first intended contribution of this paper is to take the analysis further and determine the \emph{finite size scaling} (FSS) behavior, i.e. explicitly characterize the first correction term $f_2$ and establish an error bound to the effect of
\begin{eqnarray}\label{eq:fss}
P_\ex(n, rn^{-1/2}) = f_1(r) + f_2(r) n^{-1/6} + O(n^{-1/4+\eps}) 
\end{eqnarray}
for any $\eps>0$ (see Section \ref{sec:introexit}). 

The main result (\ref{eq:fss}) has direct applications to the analysis of randomized algorithms in large combinatorial problems, where density dependent chains describing position in a state-space $\sfS$ naturally appear. One is interested in the probability of such an algorithm failing (or halting), which corresponds to the exit of the chain from some domain, leading to a phase transition (failure vs. success) for the initial data. The expansion (\ref{eq:fss}) allows one to obtain detailed information about the FSS in quite some generality from essentially the solution of three differential equations. Obtaining such detailed information has only been possible so far in special cases, for example the Erd\"os-Renyi random graph at the criticality (cf.  \cite{JKLP93}, \cite{ABG09} and references therein), even though cutoff phenomena like the one just described have been intensely studied. Cf. \cite{Fri99} for existence of sharp thresholds for graph properties; \cite{Wor95} for location of thresholds for randomized agorithms;
\cite{Wil02} for size of a phase transition window (`Harris criterion'); \cite{Di96} and others  for cutoffs for mixing times. 

The application described in the previous paragraph is illustrated here with a specific example. Section 5 develops a framework, based around Proposition \ref{lem:hyperasymptotic}, for the study of the `peeling' algorithm in large random hypergraphs. Within this framework the main result of this paper is applied to establish the FSS behavior in the phase transition for the existence of a 2-core in an ensemble of particular interest (see Section \ref{sec:introhyper}). This result has important consequences for information transmission over noisy channels, which are explained in Remark \ref{rem:importance}. Morevoer, other ensembles of importance can be now rather straightforwardly analyzed using this framework (see Remark \ref{rem:final}).

The result (\ref{eq:fss}) was derived in \cite{DM08} for a special case (a particular Markov chain), with $\sfD = \ints_+\times\ints$ and one critical time ($N=1$). In the general setting considered here new methods are needed to handle the difficulties coming from handling several critical times, exit through different faces of $\sfD$, and the chain not being density dependent outside $\sfD^\circ$. The latter is not an artificial condition and occurs in applications, where the algorithm failure corresponds to the chain reaching a coffin state at the boundary of $\sfD$, precluding extension of the kernel $W_n$. This is dealt with here by introducing another Markov chain with kernel $\hW(\cdot|\x) = W(\cdot|\sfK(\x))$ (extension of the smooth kernel $W$ through projection $\sfK$ onto $\sfD$) and comparing the vector field associated with $\hW$ to the smooth extension (from $\sfD^\circ$) of the vector field associated with $W$ (see Lemma \ref{lem:solutionproperties}). In turn the existence of several critical times prevents computation of $P_\ex$ by conditioning on not-exiting near previous critical times, as such conditioning changes the distribution of the process. Instead, multivariate Taylor expansion of standard normal distribution in $\RR^N$ is used and it is shown that the dependent terms corresponding to the influence of exit near previous critical times are small enough (see Proposition \ref{prop:2}). This also leads to a certain simplification and a sharper error estimate in the asymptotic expansion as compared with \cite{DM08}. 

The remainder of Section 1 is devoted to the rigorous statements of the main result (\ref{eq:fss})  (in Section \ref{sec:introexit}) and the FSS for the 2-core problem (in Section \ref{sec:introhyper}).

\subsection{Density dependent Markov chain exiting a polytope}\label{sec:introexit}

Let $\sfD \equiv \{\x \in \RR^d : \sfnr_i\cdot\x \geq \sfg_i \;\mathrm{for}\; i=1,\ldots,\sfN \}$, where $\sfnr_i$ are unit vectors and $\sfg_i \in \RR$, be a fixed non-degenerate convex polytope in $\RR^d$ and let $\{\sfWD\}$ be a family of kernels defining transition probabilities on some countable set $\sfS$ as in (\ref{def:Wn}) and satisfying the following two conditions:
\begin{enumerate}
\item increments are uniformly bounded, i.e. there exists a finite constant $\sfk_1$ such that for all $n$, $\z \in \sfS$ and $\D\in\RR^d$ with $||\D|| > \sfk_1$ we have $\sfWD(\D| \z) = 0$;
\item there exists a probability kernel $W(\cdot|\cdot): \RR^d\times\sfD \rightarrow [0,1]$, such that for each $\D \in \RR^d$ the functions $\{ \x \mapsto W(\D|\x) : \D\in\RR^d\}$ are twice differentiable in $\sfD$ (including the boundary) with (uniformly in $\D$) Lipschitz continuous derivatives and such that for some constants $\sfk_2>0$ and $\maly < \tfrac12$ we have
\begin{eqnarray}\label{ass:dd}
\l| \sfWD(\cdot | \z) -  W(\cdot | n^{-1}\z) \r|_{\sTV} \leq \sfk_2 n^{-1+\maly}
\end{eqnarray}
for all $\z \in \sfS\cap (n\sfD^\circ)$, where $n\sfD^\circ = \{\x\in\RR^d : n^{-1}\x \in\sfD^\circ\}$ and $||\cdot||_{\sTV}$ denotes the total variation distance.
\end{enumerate}
Let $\rho \mapsto \y_\rho$ be a twice differentiable map with bounded derivatives from $(-\mala,\mala)$ into $\sfD^\circ$. Let $\y(\cdot)$ be a twice differentiable (with Lipschitz continuous second derivative) solution of (\ref{eq:ODE}) on the interval $\th \in [0,\ovth]$ with $\y(0) = \y_0$. Suppose there exists a finite subset $\Theta \equiv \{\th^1_c,\ldots, \th^N_c\} \subset (0,\ovth)$ such that $\sfd(\y(\th)) \equiv \min_i (\sfnr_i\cdot\y(\th) - \sfg_i) > 0$ for all $\th \in [0,\ovth]\backslash\Theta$, whereas for each $1 \leq i \leq N$ there exists a unique $j(i)$ such that setting $\sfmr_i \equiv \sfnr_{j(i)}$ and $\sfd_i \equiv \sfg_{j(i)}$ we have
\begin{eqnarray}\label{tangencyconditions}
\sfmr_i\cdot\y(\thic) = \sfd_i\;, \qquad
\sfmr_i\cdot \frac{\did \y}{\did \th} (\thic) = 0\;, \qquad\mbox{and}\qquad
\sfmr_i\cdot \frac{\did^2 \y}{\did \th^2} (\thic) > 0.  
\end{eqnarray}
To allow for a random perturbation of the initial state $\y_\rho$, let $\{\zrn: n, \rho \}$ be a family of random variables taking values in $\sfS\cap (n\sfD^\circ)$ and such that for some positive constants $\sfk_3$, $\sfk_4$, $\sfk_5$ and $\sfk_6$ the following hold uniformly in $n$ and $\rho$:
\begin{eqnarray}
\P \{ \l| \zrn - \E \zrn \,\r| \geq r \} &\leq& \sfk_3 e^{-r^2/(\sfk_4 n)}\label{ass:exp}\\
|| \E\zrn - n \y_\rho || &\leq& \sfk_5 \label{ass:mean}\\
\sup_{\U \in {\cM}_{N, d}} \,\sup_{\x \in \RR^d} 
\ln \P\{\U (\zrn - \E \zrn) \leq \x \} - \P\{ n^{1/2}\U \vze_\rho \leq \x \} \rn 
&\leq& \sfk_6 n^{-1/2}, \label{neargaussian}
\end{eqnarray}
where $\cM_{N, d}$ denotes the collection of real $N\times d $ matrices and $\vze_\rho$ is a Gaussian vector in $\RR^d$ with zero mean and positive semi-definite (p.s.d.) covariance matrix $\Q_\rho$, such that $\rho \mapsto \Q_\rho$ is a Lipschitz continuous map (for the operator norm). 
Note that, in particular, $\zrn = n\y_\rho$ satisfies this with $\Q_\rho = 0$.

For each $n$ and $|\rho| < \mala$ let $\Pnr$ denote the law of the Markov chain $\proc{\z}$ started at $\z(0) =\zrn$ and with transition probabilities given by $\sfWD$. The following is the formal statement of the result announced as (\ref{eq:fss}).

\begin{thm}\label{thm:markovexit}
Let $r\in \RR$, $\rho_n =  rn^{-1/2}$ and assume that $\bS$, defined in (\ref{def:bSinthm}), is non-singular. Then for any $\eta < 1/4\wedge(1/2-\maly)$ we have
\begin{eqnarray*}
P_\nex(n,\rho_n)\equiv \P_{n,\rho_n}\lbr \min_{\ta\leq \ovn}\sfd(\ztau) > 0 \rbr = \Phi(r\aa) - \bb^\dag \,\nabla \Phi(r\aa)\; \Omega \, n^{-1/6}  + O(n^{-\eta}),
\end{eqnarray*}
where $\Phi(\cdot)$ denotes the standard normal distribution function in $\RN$, $\aa = \bS^{-\frac{1}{2}}\bG$ and $\bb = \bS^{-\frac{1}{2}}\bLam$ for the $N\times N$ p.s.d. symmetric matrix $\bS$ and $N\times 1$ vectors $\bG$ and $\bLam$ with entries given by ($1 \leq i \leq j \leq N$)
\begin{eqnarray}
\relax [\bS]_{i j} &\equiv& \sfm_i \,\Q(\thic)\, (\B_{\thic}(\thjc))^\dag \,\sfmr_j  
\label{def:bSinthm}\\
\relax [\bG]_i &\equiv& \sfm_i\,\B_0(\thic) \left.\frac{\did \y_\rho}{\did \rho} \right|_{\rho=0} \label{def:bGinthm}\\
\relax [\bLam]_i &\equiv& \lsq\sfmr_i\cdot\frac{\did^2 \y}{\did \th^2}(\thic) \rsq^{-1/3}\lsq \sfm_i\,\G(\y(\thic))\,\sfmr_i \rsq^{2/3} \label{def:bLam},
\end{eqnarray}
where $\G(\x)$ is the covariance matrix of $W(\cdot|\x)$ defined in (\ref{def:G}), and where $\B$ and $\Q$ satisfy the system (\ref{ode:b}-\ref{ode:var}) at $\rho=0$. Finally, the finite constant $\Omega$ is given by the integral
\begin{eqnarray}\label{def:Omega}
\Omega \equiv \int_0^\infty [1 - \cK(z)^2] \did z \, ,
\end{eqnarray}
where
\begin{eqnarray}\label{def:cK}
\cK \equiv \frac 12 \int_{-\infty}^\infty \frac{\Ai(iy)\Bi(2^{1/3}z+iy) - \Ai(2^{1/3}z+iy)\Bi(iy)}{\Ai(iy)} \did y
\end{eqnarray}
and $\Ai(\cdot)$, $\Bi(\cdot)$ are the Airy functions (as defined in \cite[p. 446]{AS64}).
\end{thm}

\begin{remark}
Theorem \ref{thm:markovexit} can be extended, as follows, to provide some information in the case when $\bS$ is singular, which can occur when the Markov chain is degenerate. Write 
$\bS = \bO \,{\rm diag}(\la_1, \ldots, \la_N)\, \bO^\dag$, where $\bO = (\bv_1, \ldots, \bv_N)$ is an orthogonal matrix with $\{\bv_i\}$ the eigenvectors of $\bS$ and $\{\la_i\}$ are the corresponding eigenvalues of $\bS$ in decreasing order. Let $k \equiv \min\{i: \la_i = 0\} < N$ and define $\cV = \langle \bv_1,\dots,\bv_k \rangle$, the span of $\{\bv_1,\ldots,\bv_k\}$. Setting  $\bV \equiv (\sqrt{\la_1}\bv_1, \ldots, \sqrt{\la_k}\bv_k)$ and $\vze$ is a standard Gaussian vector in $\RR^k$, straightforward calculus shows that the function
$\widehat{\Phi}(\cdot) \equiv \P(\bV \vze \leq \cdot)$ is smooth on $\RR^N\backslash \cV$. It follows (by modifying the last part of the proof of Proposition \ref{prop:2} to allow for singular $\bS$ by projecting $\{\bg_k+\hat{\bg}_k\}$ (there) onto $\cV$, and generalizing the estimate in Lemma \ref{lem:gaussianbound}), that for all $r\in\RR$ such that $r\bG \notin \cV$
\[
P_\nex(n,\rho_n) = \widehat{\Phi}(r\bG) - \bLam^\dag \,\nabla  \widehat{\Phi}(r\bG)\; \Omega \, n^{-1/6}  + O(n^{-\eta}),
\]
where $\eta$ is like in Theorem \ref{thm:markovexit}.
In particular,  if $(r\bG + \cH) \cap \cV = \emptyset$, where $\cH$ denotes the negative orthant of $\RR^N$, then actually
\[
P_\nex(n,\rho_n) = O(n^{-\eta}).
\]

It is also straightforward to modify the proof to yield an expansion for the probability of exiting near a particular critical time, or equivalently through a particular face of $\sfD$. Finally, the result extends to domains with smooth boundaries, though the proof then requires analysis in boundary layer coordinates near critical times. The choice of polytope domains here is motivated by applications and clarity of presentation. 
\end{remark}

\subsection{Finite size scaling for the 2-core in irregular hypergraph ensembles} \label{sec:introhyper}

A $2$-core in a hypergraph is the unique maximal sub-hypergraph such that no hypervertex in it has degree less than $2$. The size of a $2$-core is defined to be the number of hyperedges it contains. In the following, we occasionally refer to the $2$-core as the `core'. Its existence can be established through recursively deleting all vertices of degree less than $2$ together with the associated hyperedges, which is the basis for the randomized \emph{peeling} or \emph{leaf removal} algorithm proposed by \cite{KS81}. Algorithms equivalent to this one, as well as $k$-cores of graphs and hypergraphs, have been studied as combinatorial objects, in connection to the XOR-SAT problem, and in the analysis of low-density parity-check codes over the binary erasure channel (cf. \cite{DM08} and references therein for more background). 

We are interested in the ensemble $\cG \equiv \cG(n,m,\sfv_n)$. $\cG$ is a collection of labelled hypergraphs in which each hyperedge $\al$ is a list (of size $j_\al\leq L < \infty$ uniformly in $n$ ) of \emph{not necessarily distinct} hypervertices chosen uniformly at random and with replacement. For each $n$ and $j$ ($3 \leq j \leq L$) there are $\sfv_n(j)$ hyperedges of size $j$ and we have $\sum_{j=3}^L\sfv_n(j)=n$. Moreover, we assume that there is a non-negative vector $\vn_0 = (v_{0,3},\ldots,v_{0,L})$ such that $|\sfv_n(j)-nv_{0,j}| \leq c_0 < \infty$ uniformly in $n$. (See Section \ref{sec:ensemble1} for more details.)

We are interested in the probability $P_\nocore(n,\rho)$ that a random graph $G$ in the ensemble $\cG(n,\fl n\rho \fr, \sfv_n)$ does not have a 2-core. Define $\sfV(x) \equiv \sum_{j=3}^Lj v_j x^{j-1}$ and let
\begin{eqnarray}
\rho_c \equiv \inf\{\rho>0 : \sfV(\ze)/\rho > -\log(1 - \ze) \quad\forall\,\ze \in (0,1) \}\wedge 1.
\end{eqnarray}
Consequently the set $Z \equiv \{\ze\in(0,1) : \sfV(\ze)/\rhoc = -\log(1-\ze) \}$ is non-empty and finite. We then write $Z = \{\ze_1,\ldots,\ze_{|Z|} \}$ with $\ze_i < \ze_j$ for $i<j$ and assume that 
\begin{eqnarray}
\sfV''(\ze)/\rho_c < (1-\ze)^{-2}\qquad\mbox{for all}\qquad\ze\in Z. \label{ass:ensemble1}
\end{eqnarray}
Let $\theta(\cdot)$ be the inverse function of the monotonically decreasing bijection $\ze: [0,1] \rightarrow [0,1]$, given by $\zeta(\th) \equiv \exp\{ - \int_0^\th \eta(s)^{-1} \did s \}$. Here $\eta(s) = \sum_{j=3}^Ljy_j(s)$ and $\y(\cdot)$ is the solution of the system (\ref{eq:ODE}) with
\begin{eqnarray}
\F(\x) &=&  \ll  -1 + \sfR'(1)(\pp_1-\pp_0)\,, - \sfR'(1)\pp_{1} , -\ss_3\,,\;\;\dots,\;\;\; -\ss_L  \rr^\dag\label{def:F}\\
\y_0 &=& \ll \mu e^{-\mu/\rhoc}, \; \rho_c (1-e^{-\mu/\rhoc}) - \mu e^{-\mu/\rhoc},\; \vn\,^\dag \rr^\dag,\label{def:y0}
\end{eqnarray}
where $\mu = \sfV(1)$, \;$\sfR(\xi) = \sum_{j=3}^L\ss_j\xi^{j-1}$ and where
\begin{eqnarray}
\pp_0 = \frac{x_1}{ \sum_{j=3}^Ljx_j},\quad\pp_1 = \frac{x_2\la^2}{(e^\la-1-\la)\sum_{j=3}^Ljx_j}, \quad
\ss_j = \frac{jx_j}{\sum_{j=3}^Ljx_j}\;\;\mbox{for}\;\; 3\leq j\leq L,
\end{eqnarray}
and $\la$ is the unique positive solution of $\;\la(e^\la-1)(e^\la-1-\la)^{-1} = (\sum_{j=3}^Ljx_j-x_1\vee 0)x_2^{-1}$.

\begin{thm}\label{thm:hypergraph}
Let $\rho_n = \rhoc + rn^{-1/2}$ and suppose (\ref{ass:ensemble1}) holds. Then for any $\eta < 1/4$
\begin{eqnarray*}
P_\nocore(n,\rho_n) = \Phi(r\aa_h) - \bb_h^\dag \,\nabla \Phi(r\aa_h)\; \Omega \, n^{-1/6}  + O(n^{-\eta}),
\end{eqnarray*}
where $\Phi$ and $\Omega$ are defined in Theorem \ref{thm:markovexit}, while $\aa_h$ and $\bb_h$ are respectively $\aa$ and $\bb$ of Theorem \ref{thm:markovexit} where one uses $N = |Z|$, $\thic = \th(\ze_{N-i})$, $\sfmr_i = \vec{\sfe}_1 \equiv (1,0,\ldots,0)^\dag$, $\y_0$ defined in (\ref{def:y0}), p.s.d. symmetric matrix $\Q_0$ with entries ($3 \leq i \leq j \leq L$)
\begin{eqnarray}
\left\{ \begin{array}{rcll}
Q_{0,11} &=& \rho_c \gam e^{-2 \gam} (e^\gam - 1 + \gam - \gam^2) \\
Q_{0,12} &=& - \rho_c \gam e^{-2 \gam} (e^\gam - 1 - \gam^2)\\
Q_{0,22} &=&  \rho_c\, e^{-2 \gam} [ (e^\gam - 1) + \gam(e^\gam - 2) - \gam^2(1+ \gam)]\\
Q_{0,ij} &=& 0,
\end{array}\right.
\end{eqnarray}
vector $\F$ defined in (\ref{def:F}), and p.s.d. symmetric matrix $\G(\x)$ 
with entries ($3\leq i \leq j \leq L$)
\begin{eqnarray}
\left\{ \begin{array}{rcl}
G_{11}(\x) &=& \sfR'(1)(\pp_0+\pp_1) + [\sfR''(1) - \sfR'(1)^2](\pp_1-\pp_0)^2\\
G_{12}(\x) &=& \sfR'(1)\pp_1 - [\sfR''(1) - R'(1)^2](\pp_1-\pp_0)\pp_1\\
G_{1j}(\x) &=& \ss_j(j-1-\sfR'(1))(\pp_1-\pp_0)   \\
G_{22}(\x) &=& \sfR'(1)\pp_1 + [\sfR''(1) - R'(1)^2]\pp_1^2\\
G_{2j}(\x) &=& - \ss_j(j-1-\sfR'(1))\pp_1         \\
G_{ij}(\x) &=& - \ss_i\ss_j + \ss_j\1_{i=j}.
\end{array}\right.
\end{eqnarray}
\end{thm}

\begin{remark}\label{rem:importance}
Theorem \ref{thm:hypergraph} should be compared to Theorem 1 of \cite{DM08}, where the hyper-edges are restricted to have the same common size $l\leq 3$.  For such hypergraphs it can be shown that the set $Z$ has only one element, i.e. the critical solution $\y$ has only one critical time, whereas for irregular ensembles like $\cG(n,m,\sfv_n)$ there can be several 
critical times. Thus the general Theorem \ref{thm:markovexit} is necessary for the proof of Theorem \ref{thm:hypergraph}. See Remark \ref{rem:final} for more details on the analysis of other initial ensembles.

The framework of Section \ref{sec:hypergraph} and Theorem \ref{thm:hypergraph} have important consequences for coding theory, which are explained here briefly (cf. \cite{RU08} for a comprehensive account). In the analysis of low-density parity-check (LDPC) codes over the binary erasure channel (BEC), there appears a natural hypergraph structure, where the ratio $\rho = m/n$ is the transmission rate (equal to the noise level for the channel) and $n$ is the transmitted message length. Peeling algorithm on this hypergraph corresponds to a certain decoding strategy and the presence of a 2-core implies a decoding failure. Design of LDPC codes translates into the construction of hypergraph ensembles with maximal $\rho$ such that the probability of decoding failure vanishes as $n$, the message length, tends to infinity. Such $\rho$ is bounded above by the theoretical limit for the transmission rate over the channel, the so-called Shannon capacity, and the LDPC codes have been shown to achieve rates close to this limit.

Even though the probability of failure vanishes as $n\rightarrow\infty$ for any $\rho<\rhoc$ (some critical threshold), it is still significant for message lengths encountered in practice, which could be $n = 1,000$ to $10,000$, and so ultimately better estimates of this probability are necessary for successful code design. This motivated the conjecture in \cite{AMRU04} about the finite size scaling behavior at this phase transition, which was subsequently rigorously proved in \cite{DM08} for the ensembles of regular hypergraphs.

The overall goal, however, is to design finite LDPC code ensembles that minimize the decoding error (cf. \cite{AMU07} for a systematic approach to this task). This optimization gives rise to irregular hypergraph ensembles (that is with hyperedges of different sizes) with core size (proportion of hyperedges contained in the core) concentrating near several distinct values as $n$ tends to infinity. This corresponds to the case of $|Z| = N > 1$ and is a consequence of the so-called matching condition (cf. \cite[Section 3.14.4]{RU08}), which translates into an integral condition for the generating function $\sfV$, in terms of which the set $Z$ is defined. It is thus necessary to deal both with irregular ensembles and with the case of $N > 1$, both of which goals are accomplished here.
\end{remark}

This paper is organized as follows. The proof of Theorem \ref{thm:markovexit} spans Sections \ref{sec:asymptotickernel}-\ref{sec:weak}. Section \ref{sec:asymptotickernel} is devoted to preliminary results dealing with the loss of the density dependent property of the kernel at the boundary of $\sfD$ and establishing properties of the associated differential equations.
Section \ref{sec:coupling} is an adaptation of the coupling arguments from \cite{DM08}, culminating with a coupling to a Brownian motion with quadratic drift. Section \ref{sec:weak} contains the new arguments used to handle multiple critical times. Finally, Section \ref{sec:hypergraph} develops the framework for analyzing hypergraph 2-cores, where Proposition \ref{lem:hyperasymptotic} is of particular interest, and concludes with the proof of Theorem \ref{thm:hypergraph}.

\section{Asymptotic kernel and the associated vector field}\label{sec:asymptotickernel}

This section is devoted to preliminary results about the properties of the system (\ref{ode:mean}-\ref{ode:var}) needed for our analysis. Throughout the paper we denote the open interior of a set $\sfA \subset \RR^d$ by $\sfA^\circ$ and define the $\eps$-thickening of that set by $\sfA^\eps \equiv \{\x\in\RR^d : \dist(\x,\sfA) < \eps \}$. We let $\sfd(\z) \equiv \min_i (\sfnr_i\cdot\z - \sfe_i)$ and it is easy to check that this equals $\dist(\z,\RR^d\backslash\sfD)$ or $-\dist(\z,\sfD)$ depending on whether $\z$ lies in $\sfD$ or not. Greek letters $\ta$ and $\si$ are non-negative integers denoting discrete time, whereas $t$ and $s$ are non-negative real numbers denoting continuous time. Finally, $\{c_i\}$ are positive constants, possibly distinct in each proof.

For $\x \in \sfD$ let $\A(\x)$ be a matrix of partial derivatives of $\F$ at $\x \in \sfD$, i.e. $[\A]_{ab}=\partial_{x_b} F_a(\x)$, which is indeed well-defined for all $\x \in \sfD$ as $\F$ certainly admits a differentiable extension to $\sfD^\delta$. Also, define $\G(\x)$ to be the covariance of the increment $\D$ with distribution $W(\cdot|\x)$, i.e. the p.s.d. symmetric matrix
\begin{eqnarray}
\G(\x) \equiv \sum_{\D} \D \,(\D)^\dag W(\D|\x) - \F(\x)\, \F(\x)^\dag.\label{def:G}
\end{eqnarray}

Let $\sfK: \RR^d \rightarrow \sfD$ denote the projection onto $\sfD$, which was assumed to be convex and non-degenerate. We extend $\x \mapsto W(\cdot|\x)$ as well as $\F$, $\A$, and $\G$ to the whole of $\RR^d$ defining $W(\cdot|\x) \equiv W(\cdot|\sfK(\x))$, $\F(\x) \equiv \F(\sfK(\x))$, etc. Note that $W(\cdot|n^{-1}\,\cdot)$ is now a transition kernel on $\RR^d$ and that for all $\x\in\RR^d$ we still have $\F(\x)$ and $\G(\x)$ as the mean and variance, respectively, of the increment $\D$ with distribution $W(\cdot|\x)$. However, $\A(\x)$ is now only a matrix of derivatives of $\F$ for $\x \in \sfD^\circ$. Finally, $\F$, $\A$, and $\G$ are Lipschitz continuous on $\RR^d$, as they were on $\sfD$.

For each $\rho$ with $|\rho|<\mala$, we are interested in the following system of ordinary differential equations
\begin{eqnarray}
\frac{\did \y}{\did \th}(\th) &=& \F(\y(\th)) 
\qquad \th\geq 0, \;\y(0) = \y_\rho  \label{ode:mean} \\
\frac{\did \B_\ze}{\did \th}(\th) &=& \A(\y(\th)) \B_\zeta(\th)
\qquad \th\geq\zeta\geq 0, \;\B_\ze(\ze) = \ind \label{ode:b}\\
\frac{\did \Q}{\did \th}(\th) &=& \G(\y(\th)) + \A(\y(\th)) \Q(\th) + \Q(\th) \A(\y(\th))^\dag
\qquad \th\geq 0, \;\Q(0) = \Q_\rho. \label{ode:var}
\end{eqnarray}
The solution to equation (\ref{ode:mean}) is the fluid limit mentioned in the introduction and the solution to (\ref{ode:var}) gives the covariance matrix of the diffusive limit. The matrix $\B_\ze(\th)$ is the matrix of partial derivatives of the function $\y$ at time $\th$ with respect to the initial condition specified 
at time $\ze$.

\subsection{Properties of the ODEs}
The following lemma allows us deal with the fact that $W(\D| \cdot)$ is not differentiable at $\x \in \pip\sfD$.

\begin{lemma}\label{lem:solutionproperties}
The following hold with $\cA \equiv [0,\ovth]\times(-\mala,\mala)$. 
\begin{enumerate} 
\item There is a unique solution $(\th,\rho) \mapsto \y(\th,\rho)$, $(\th,\rho) \mapsto \B_{\ze}(\th,\rho)$ and $(\th,\rho) \mapsto \Q(\th,\rho)$ of (\ref{ode:mean}-\ref{ode:var}) on $\cA$. These functions are Lipschitz continuous there and moreover there is a uniform (for $(\th,\rho)\in\cA$) bound on the operator norms $||\B_{\ze}(\th,\rho)||$, $||\B_{\ze}(\th,\rho)^{-1}||$ and $||\Q(\th,\rho)||$.

\item There are positive constants $\ka_1$ and $\ka_2$ such that for all $(\th,\rho) \in \cA$
\begin{eqnarray}
\l| \y(\th,\rho) - \y(\th,0) - \rho\, \B_0(\th,0)
\left. \frac{\did\y_\rho}{\did\rho} \right|_{\rho=0}\r| 
\leq \ka_1|\rho|^2 + \ka_2\mu_\rho |\rho|, \label{eq:piprho}
\end{eqnarray}
where $\mu_\rho$ is the Lebesgue measure of the set $\{s \in [0,\th]: \y(s,\rho) \notin \sfD\}$.  Hence, a fortiori, the following partial derivative exists
\begin{eqnarray}
\frac{\pip \y}{\pip \rho}(\th,0) \equiv \B_0(\th,0)\left.\frac{\did \y_\rho}{\did \rho}\right|_{\rho=0}
\end{eqnarray}
\end{enumerate}
\end{lemma}
\begin{proof} (a) Since $\F(\cdot)$ is bounded and Lipschitz continuous in $\sfD$ and since for any $|\rho| < \mala$  the inital condition $\y_\rho$ is in $\sfD^\circ$, it follows that the solution $\y(\cdot,\cdot)$ of (\ref{ode:mean}) exists and is unique in $\cA$. Applying Gronwall's lemma we can deduce that the solution to (\ref{ode:mean}) must be Lipschitz continuous with respect to its initial condition, so since $\rho\mapsto\y_p$ is differentiable, it also follows that $\y(\cdot,\cdot)$ is Lipschitz continuous in both coordinates.

Next, recall that $\A(\cdot)$ and $\G(\cdot)$ are Lipschitz continuous, so in view of the established properties of $\y(\cdot,\cdot)$  we can deduce the same for $(\th,\rho) \mapsto \A(\y(\th,\rho))$ and $(\th,\rho) \mapsto \G(\y(\th,\rho))$. Moreover, we have assumed $\rho \mapsto \Q_\rho$ to be Lipschitz continuous. It follows that the linear ODEs (\ref{ode:b}) and (\ref{ode:var}) have unique Lipschitz continuous and thus bounded solutions on $\cA$. In fact, suppressing the dependence on $\rho$, we have the explicit respresentation
\begin{eqnarray}
\Q(\th) = \B_0(\th)\Q(0)(\B_0(\th))^\dag + 
\int_0^{\th} \B_{\ze}(\th) \G(\y(\ze)) (\B_{\ze}(\th))^\dag \did \ze.\label{def:Q}
\end{eqnarray} 

It remains to be shown that $|| \B_\ze(\th,\rho)^{-1}||$ is uniformly bounded. This follows from the observation that ${\rm det}\, \B_\ze(\ze) = 1$, while
\[
\frac{\did ({\rm det} \B_\ze(\th))}{\did \th} = ({\rm det} \B_\ze(\th)) \;{\rm tr}[\A(\y(\th,\rho))]
\]
and the entries of $\A(\y(\th,\rho))$ are uniformly bounded on $\cA$.

(b) In order to establish (\ref{eq:piprho}), let $\F_e(\cdot)$ be a twice differentiable (with Lipschitz continuous derivatives) extension of $\F(\cdot)$ from $\sfD$ to $\RR^d$. It follows from an elementary result on differential equations (c.f. \cite[Corollary 4.1]{Har64}) that the differential equation
\begin{eqnarray}
\frac{\did \y}{\did \th}(\th) = \F_e(\y(\th)) \qquad \th\geq 0, \quad\y(0) = \y_\rho 
\end{eqnarray}
has a unique solution $\y_e$ with $(\th,\rho) \mapsto \y_e(\th, \rho)$ twice continuously differentiable on $\cA$ and moreover
\begin{eqnarray}
\frac{\pip \y_e}{\pip \rho}(\th,\rho) = \B_0(\th,\rho)\frac{\did \y_\rho}{\did \rho}
\end{eqnarray}
In particular, $\y(\th,0) = \y_e(\th,0)$, since we assumed that $\sfd(\y(\th,0)) \geq 0$ implying $\y(\th,0)\in\sfD$, and for some constant $\ka_1 > 0$ and all $\th\in[0,\ovth]$ Taylor expansion gives
\begin{eqnarray}
\l| \y_e(\th,\rho) - \y(\th,0) - \rho \frac{\pip \y_e}{\pip \rho}(\th,0) \r| 
\leq \ka_1 |\rho|^2. \label{eq:piprho2}
\end{eqnarray}
Now fix $\rho$ and $[s,u] \subset [0,\ovth]$. If $\y(\th,\rho) \in \sfD$ for all $\th \in [s,u]$, then $\F(\y(\th,\rho)) = \F_e(\y(\th,\rho))$ and Gronwall's lemma gives
\begin{eqnarray}
\l| \y(u,\rho) - \y_e(u,\rho)  - [\y(s,\rho) - \y_e(s,\rho)] \r| \leq
\l| \y(s,\rho) - \y_e(s,\rho) \r| e^{(u-s)c_2}
\end{eqnarray}
where $c_2>0$ is the Lipschitz constant for $\F_e$. On the other hand, we can also use the identity $\F(\x) = \F_e(\sfK(\x))$ to obtain
\begin{eqnarray*}
\l| \y(u,\rho) - \y_e(u,\rho)  - [\y(s,\rho) - \y_e(s,\rho)] \r| &\leq&
\l| \int_s^u \lsq\F_e(\sfK[\y(s',\rho)]) - \F_e(\y_e(s',\rho))\rsq \did s' \r| \\
&\leq& (u-s) c_2 \sup_{s'\in [s,u]} \l| \sfK[\y(s',\rho)]) - \y_e(s',\rho) \r| \\
&\leq& c_3 (u-s)|\rho| \,,
\end{eqnarray*}
where the last inequality follows by Lipschitz continuity of both $\y$ and $\y_e$, together with an application of the triangle inequality
\[
\l| \sfK[\y(s',\rho)]) - \y_e(s',\rho) \r| \leq
\l| \sfK[\y(s',\rho)]) - \sfK[\y(s',0)] \r| + \l| \y(s',0) - \y_e(s',\rho) \r|,
\]
and our assumption of $\y(s',0) \in \sfD$ for $s' \in [0,\ovth]$. Since none of the constants $c_i$ depend of $\rho$ or $[s,u]$, by continuity of $\th \mapsto \y(\th,\rho)$ it follows that for all $(\th,\rho) \in \cA$
\[
\l| \y(\th,\rho) - \y_e(\th,\rho) \r| \leq c_4e^{\ovth c_2} \mu_\rho|\rho|,
\]
which together with (\ref{eq:piprho2}) gives the inequality (\ref{eq:piprho}).
\end{proof} 

The following simple estimate will be used to approximate the mean and covariance matrix of the chain near the critical times.

\begin{lemma}\label{lem:figi}
Let $\thc \in (0,\ovth)$ and fix some positive integer $0 \leq \nu < \floor{n\thc}$. Then there exist positive constants $\ka_1$ and $\ka_2$ such that for any integer $\ta$ with $\nu \leq \ta \leq 2 \floor{n\thc} - \nu$
\begin{eqnarray}
\l| \sum_{\si=\nu}^{\ta-1}\F(\y(\si/n,\rho)) - \F_0(\ta -\nu)/n - 
\F_1 \int_{\nu-0.5}^{\ta -0.5}(s/n - \thc) \did s \r|
&\leq& \ka_1(\,|\nu/n - \thc|^2 + |\rho|\,)\,( \ta - \nu ) \label{eq:Fi}\\
\l| \sum_{\si = \nu}^{\ta-1} \G(\y(\si/n,\rho)) - \G_0(\ta - \nu) \r| &\leq&  
\ka_2(\,|\nu/n - \thc| + |\rho|\,)\,( \ta - \nu ) \label{eq:Gi}
\end{eqnarray}
where $\F_0 = \frac{\did \y}{\did \th}(\thc,0) = \F(\y(\thc,0))$,  
$\F_1 = \frac{\did^2 \y}{\did \th^2}(\thc,0) = \A(\y(\thc,0)) \F(\y(\thc,0)) $ and $\G_0 = \G(\y(\thc,0))$.
\end{lemma}
\begin{proof}
By the Lipschitz continuity of $(\th,\rho) \mapsto \y(\th,\rho)$, $\x \mapsto \F(\x)$ and $\x \mapsto \G(\x)$ it is immediate that
\begin{eqnarray*}
  \l| \F(\y(\si/n,\rho)) - \F(\y(\si/n,0)) \r| &\leq& c_0 |\rho| \\ 
  \l| \G(\y(\si/n,\rho)) - \G(\y(\th,0)) \r| &\leq& \ka_2(\,|\si/n - \th| + |\rho|\,)
\end{eqnarray*}
(where the second inequality is for the operator norm) from which (\ref{eq:Gi}) follows.
To obtain (\ref{eq:Fi}) recall that $\th \mapsto \y(\th,0)$ was assumed to be twice differentiable with Lipschitz continuous second derivative. Consequently the following Taylor expansion holds
\begin{eqnarray}
\l| \frac{\did \y}{\did \th}(\ze,0) - \frac{\did \y}{\did \th}(\thc,0) - (\ze - \thc) \frac{\did^2 \y}{\did \th^2}(\thc,0) \r| \leq c_1 |\ze - \thc|^2
\end{eqnarray}
and taking $\ka_1 = c_0 \vee c_1$ we have
\begin{eqnarray}
\l| \sum_{\si=\nu}^{\ta-1}\F(\y(\si/n,\rho)) - \sum_{\si=\nu}^{\ta-1} \lsq \F_0 - 
\F_1 (\si/n - \thc) \rsq \r|
\leq \ka_1(\,|\nu/n - \thc|^2 + |\rho|\,)\,( \ta - \nu ).
\end{eqnarray}
Observing that $f(\si) = \int_{\si-0.5}^{\si+0.5}f(s) \did s$ for $f(s) \equiv 1$ and 
$f(s) = s/n - \thc$ yields the desired conclusion.
\end{proof}



In the next lemma we establish certain discrete approximations to the solutions of the differential equations (\ref{ode:mean})-(\ref{ode:var}). These approximations will be related to the mean and covariance matrix of the auxiliary process to be introduced in Section \ref{sec:coupling}. 
Fix some $J_n \subset [0,n\ovth]$. Let $I_n \equiv [0,n\ovth] \backslash J_n$ and define the matrices $\tA_\ta \equiv \A(\ytnrho) \1_{\ta \in I_n}$. Further, let the matrix $\tB_{\si}^{\ta} \equiv \ind$ \;for $\si > \ta$ while for $0 \leq \si \leq \ta$ let
\[
\tB_{\si}^{\ta} \equiv   \ll \ind + \frac1n \tA_{\ta} \rr\cdots \ll \ind + \frac1n \tA_{\si} \rr.
\]
Finally, define the following sequences (as before $\si$ and $\ta$ are non-negative integers)
\begin{eqnarray}
\yst(\ta+1) &=& \yst(\ta) + \tA_\ta (\yst(\ta) - \ytnrho) + \F(\ytnrho)\;,\qquad  
\yst(0) = n\y_\rho\\ 
\nonumber\\
\Qst(\ta+1) &=& n \tB_0^\ta\, \Q_\rho ( \tB_0^\ta )^\dag  + 
\sum_{\si \in I_n\cap [0,\ta]} \tB_{\si+1}^\ta \, \G(\ysnrho) ( \tB_{\si+1}^\ta )^\dag, \qquad
\Qst(0) = n \tB_0^\ta\, \Q_\rho ( \tB_0^\ta )^\dag, \label{def:var}\qquad
\end{eqnarray}
where we suppressed the explicit dependce of $\yst$ and $\Qst$ on $\rho$. The following lemma is proved in Appendix \ref{app:A}.

\begin{lemma}\label{lem:meanvar}
There are positive constants $\ka_1$, $\ka_2$ and $\ka_3$ such that for all $\ta \in [0,n\ovth]$, $\ze \in [0, \ta/n]$ and $n$ large enough we have
\begin{eqnarray}
\l| n^{-1}\yst(\ta) -  \y(\ta/n, \rho) \r| &\leq& \ka_1 n^{-1} \label{eq:ystar} \\
\l| \tB_{ \ceil{n \ze} }^{\ta-1} - \B_\ze(\ta/n) \r| &\leq& \ka_2 (1 + |J_n|)n^{-1} \label{eq:tB} \\
\l| n^{-1}\Qst(\ta) - \Q(\ta/n,0) \r| &\leq& \ka_3 (1 + |J_n|)n^{-1}. \label{eq:Qstar}
\end{eqnarray}
Moreover, matrices $\{ \tB_{\si}^{\ta} : 0 \leq \si, \ta \leq n\ovth \}$ and their inverses are uniformly bounded in operator norm.
\end{lemma}

\subsection{Fluid limit with error bounds}
The following is concentration inequality for a Markov chain with a fluid limit. Its proof isd a straightforward generalization of the proof of \cite[Lemma 5.2]{DM08}. For similar results for more general discrete-time processes see \cite{Wor95} and for pure-jump Markov process see \cite{DN08}.

\begin{lemma}\label{prop:ode}
Let $\z(\cdot)$ be a Markov chain started at some initial distribution $\z(0)$ with kernel $W(\cdot|n^{-1}\cdot)$. Then there exist constants $\ka_1$, $\ka_2$ and $\ka_3$, depending only on the constants $\sfk_i$ and $\ovth$, such that for all $n$ and any $0 \leq \tau \leq \floor{n\ovth}$: 
\begin{enumerate}
\item $\z(\cdot)$ is exponentially concentrated round its mean
\begin{eqnarray}\label{est:expcon}
 \Pnr \{ \l| \ztau - \Enr\ztau \r| \geq r \}  \leq \ka_0 \exp \ll {-\frac{r^2}{\ka_1 n}} \rr, 
\end{eqnarray}
\item and the mean is close to the solution $\y(\cdot)$ of the ODE (\ref{ode:mean})
\begin{eqnarray}
\l| \Enr \ztau - n \y(\tau/n) \r| \leq \ka_2 \l|\Enr\z(0) - n\y(0) \r| + \ka_3 \sqrt{n \log n}
\end{eqnarray}
\end{enumerate}
\end{lemma}

\section{Coupling arguments}\label{sec:coupling}

In this section we construct three couplings that allow us to approximate the distribution of the $\min_{\ta\leq\ovn}\sfd(\z(\ta))$ where $\z(\cdot)$ is the Markov chain of Section \ref{sec:introexit}. The methods of proof owe a substantial debt to those of \cite{DM08}.

We define an extension $\sfWR$ of the transition kernel $\sfWD$ by
\begin{equation}
\sfWR(\Dz|\z) \equiv \lbr \begin{array}{rcl} 
\sfWD (\Dz|\z)  && \mbox{for} \quad \z \in \sfS\cap(n\sfD^\circ) \\
W(\Dz| n^{-1}\z) && \mbox{otherwise} \end{array} \right. .
\end{equation}
It follows, that when started at the same initial distribution, the Markov chain with kernel $\sfWR$ and the one with kernel $\sfWD$ coincide  until the first time they reach the boundary of $n\sfD$. Furthermore, notice that our assumption (\ref{ass:dd}) gives us the following estimate for the total variation distance
\begin{eqnarray}
\l| \sfWR(\cdot|\z) - W(\cdot | n^{-1}\zp) \r|_{\sTV} 
\leq \sfk_2n^{-1+\maly} + \sfk_7 n^{-1}||\z - \zp|| \label{eq:wclose}
\end{eqnarray}
where $\sfk_7$ is the Lipschitz constant for $W$. The following lemma allows us to reduce the  original problem to the exit problem for the Markov chain with kernel $W(\cdot|n^{-1}\cdot)$.

\begin{lemma} \label{lem:originalcoupling}
Let $\z(\cdot)$ and $\zp(\cdot)$ be Markov chains in $\RR^d$ with transition kernels $\sfWR(\cdot|\cdot)$ and $W(\cdot|n^{-1}\cdot)$ respectively and with common initial distribution $\z(0)=\zp(0)$. Then there exists a coupling $(\z(\cdot),\zp(\cdot))$ and positive constants $\ka_1$ and $\ka_2$, depending only on $\ovth$ and the constants $\sfk_i$, such that for all $r>0$
and any time $\ovth > 0$
\begin{eqnarray}
\Pn\lbr \sup_{\ta \,\leq\, \ovn} \l| \ztau - \zp(\ta)\r| > r n^{\maly} \rbr 
\leq \exp\{n^{\maly} (\ka_1 - \ka_2 r)\},
\end{eqnarray}
where $\ovn \equiv \floor{n \ovth}$.
\end{lemma}

\begin{proof}
We couple $\z(\cdot)$ and $\zp(\cdot)$ by first setting $\z(0)=\zp(0)$ and then iteratively coupling for each $\ta \geq 1$ the increments $\D\z = \z(\ta) - \z(\ta-1)$ and $\D\zp = \zp(\ta)-\zp(\ta-1)$, so that 
\begin{eqnarray}
\Pn \{\Dz \neq \D\zp | \z,\zp \} = \l| \sfWR(\cdot|\z) - W(\cdot|n^{-1}\zp) \r|_{\sTV}.
\label{eq:wcoupling}
\end{eqnarray}
Define $D_\ta(\la) \equiv \En [e^{\la Z(\ta)}]$, where
$Z(\ta) \equiv \sup_{\si\leq\ta} ||\z(\si) - \zp(\si) ||$. It is enough that we prove the bound 
$D_\ovn(\ka_2) \leq \exp\{\ka_1 n^{\maly}\}$ for some positive $\ka_1$ and $\ka_2$. This proof follows like the proof of \cite[Lemma 5.1]{DM08}, where we take $\widetilde{c} = c_0n^\mala$ and allow the constants to depend on $\ovth$.
\end{proof}

As in the comments preceding Lemma \ref{lem:meanvar}, let $J_n \subset [0,n\ovth]$, define $I_n \equiv [0,n\ovth]\backslash J_n$, and recall the definitions of $\tA_\ta$ and $\tB_\si^\ta$.
We now suppose that $|J_n| \leq n^\beta$ for some $\beta \in (0,1)$, where $|\cdot|$ denotes the Lebesgue measure on $\RR$, and define the auxiliary process
\[ 
\zp(\ta+1) = \left\{ \D_{\ta} - \tA_{\ta}\y(\tn,\rho) \right\} + \tB_{\ta}^{\ta} \zp(\ta), \qquad \zp(0) = \zrn,
\]
where $\{\D_\ta \}_{\ta=0}^\infty$ are independent random variables with distributions $W(\cdot|\y(\tn,\rho))$. Let $\Paux$ denote the distribution of $\zp(\cdot)$.

Furthermore, suppose that there exist some positive constants $c_0$ and $\gam$ and a slowly varying function $L(n)$, such that for $\z(\cdot)$, being the chain with kernel $W(\cdot|n^{-1}\cdot)$ and $\z(0)=\zrn$, we have
\begin{eqnarray}
  \Pnr \lbr \sup_{\ta \leq \ovn}||\z(\ta) - n \y(\tn,\rho)|| \geq c_0 n^\gam L(n) \rbr \leq \frac{1}{n}. \label{assum:estimate}
\end{eqnarray}
Assume also, that for some $c_1<c_0$ we have $\sfd(y(\th,\rho)) \geq c_1n^\gam L(n)$ uniformly in $\{\th : n\th \in I_n\}$ and $|\rho| < \mala$. Then the following holds.

\begin{propo}\label{prop:general}
For any $\delta > \max\{ \gam/2, 2\gam - 1, \beta +\gam - 1 \}$ there exist finite positive constants $\ka_1$ and $\ka_2$ (depending only on the constants $\sfk_i$, $c_0$ and $c_1$ above) and a coupling of the process $\z(\cdot)$, with kernel $W(\cdot|n^{-1}\cdot)$ and initial distribution $\zrn$, and 
$\zp(\cdot)$ with law $\Paux$, such that
\begin{eqnarray}
\Pnr \lbr \sup_{0 \leq \ta \leq \ovn} \l| \ztau - \zp(\ta) \r| \geq \ka_1 n^\delta \rbr \leq \frac{\ka_2}{n}
\end{eqnarray}
for all $|\rho| < \mala$ and $n$.
\end{propo}

\begin{proof}
Observe that we can couple $\z$ and $\zp$ by taking $\zp(0) = \z(0)$ and coupling the increments so that
\begin{eqnarray}
\Pnr (\D\ztau \neq \D_\ta | \calF_\ta ) = \l| W(\cdot|n^{-1}\ztau) -  W(\cdot|\ytn) \r|_{\sTV},
\end{eqnarray}
where as before $\D\ztau \equiv \z(\ta+1) - \ztau$ and $\calF_\ta$ is the $\si$-algebra generated by $\{\z(\si),\zp(\si): \si \leq \ta\}$.

We will now show that it is enough to consider the coupling only until the first time the process $\z(\cdot)$ and the solution $\y(\cdot)$ (suppressing dependence on $\rho$) are more than $c_0 n^\gam L(n)$ apart. Indeed let $\ta_\ast \leq n$ denote the first value of $\ta$ such that 
$||\ztau - n\ytn|| > c_0 n^\gam L(n)$. Then by assumption (\ref{assum:estimate}) we have
\begin{eqnarray}
\Pnr \{ \ta_\ast \leq \ovn \} \leq n^{-1} \,,
\end{eqnarray}
and it follows that to prove the statement of the proposition, it is enough to show that
\begin{eqnarray}
\sup_{\ta \leq \ovn} \Pnr \{ \ta_\ast \leq \ovn \;,\;  \l| \ztau - \zp(\ta) \r| \geq \ka_1 n^\delta \} \leq n^{-2} \,.
\end{eqnarray}
This is achieved by separately bounding the $\cF_s$-martingale $\{Z_s\}$ and the predictable process $\{V_s\}$ in the Doob's decomposition of the adapted process $N_s \equiv (\tB_0^{s-1})^{-1} (\zp(s) - \z(s)) = Z_s + V_s$, where $Z_0 = V_0 = 0$ and
\begin{eqnarray*}
\D V_{s+1} &\equiv& V_{s+1}-V_s = \En [N_{s+1}-N_s | \cF_s] = (\tB_0^{s-1})^{-1} \R(n^{-1}\z(s),\y(s/n),s) \\
\D Z_{s+1} &\equiv& Z_{s+1}-Z_s = (\tB_0^{s-1})^{-1} \{\D_s - \Dz(s) - \En[\D_s-\Dz(s) | \cF_s] \}
\end{eqnarray*}
for
\[
\R(\xp,\x,s) \equiv \F(\x) + \ind_{s \in I_n}\A(\x) [\xp -x] - \F(\xp) \, .
\]
Since $\A(\cdot)$ is Lipschitz continuous and uniformly bounded on $\RR^d$ and since it is the matrix of derivatives of $\F$ at every $\x \in\sfD^\circ$, we have
\begin{eqnarray*}
||\R(\xp,\x,s)|| &\leq& c_2 ||\xp - \x ||^2 \quad \mbox{for} \quad s\in I_n\;\;\;\mbox{and}\;\;\;\x,\;\xp \in \sfD^\circ,\\
||\R(\xp,\x,\th)|| &\leq& c_3 ||\xp - \x || \;\quad \mbox{otherwise}.
\end{eqnarray*}
Now observe that by the assumption on $\sfd(\y(\cdot))$ we have $\y(s/n)$ and $n^{-1}\z(s)$ in $\sfD$ for $s \in I_n$ and $n$ large enough as long as $\tast < n\ovth$. Thus by the bounds on $\R$
\begin{eqnarray*}
||\D V_s|| \leq \ind_{s \in I_n} c_3 ||n^{-1}z(s) - \y(s/n) ||^2 + \ind_{s \in J_n}c_4 ||n^{-1}z(s) - \y(s/n) ||
\end{eqnarray*}
leading to $||V_s|| \leq c_3 |I_n| (n^{\gam-1} L(n))^2 + c_4|J_n| (n^{\gam-1} L(n))$.

To bound $||Z_s||$ we use \cite[Lemma 5.6]{DM08} with $U_s \equiv \D Z_s$ as follows. Throughout we take advantage of the uniform bound on $||(\tB_0^{s-1})^{-1}||$. Since both $\D_s$ and $\D \z(s)$ have uniformly bounded support, we have $||U_s|| \leq 4 c_5$. Moreover on $s < \ta_\ast$
\begin{eqnarray*}
||\En[\D^\ast_s | \cF_s] || \leq c_5 \Pn(\D^\ast_s \neq 0 | \cF_s) \leq c_5 ||n^{-1}\z(s) - \y(s/n)|| \leq c_5 n^{\gam-1}L(n)
\end{eqnarray*}
implying
\begin{eqnarray*}
\Pn(||U_s|| > c_5 L(n)/n |\cF_s) \leq  \Pn(\D^\ast_s \neq 0 | \cF_s) \leq  c_5 n^{\gam-1}L(n) \;.
\end{eqnarray*}
Combining the two above estimates gives $||U_s|| \leq c_5 n^{\gam-1}L(n) + 4 c_5 \ind_{A}$ where $\Pn(A) \leq c_5 n^{\gam-1}L(n)$, making it clear that inequality \cite[(5.19)]{DM08} holds with $\Gamma = c_5 n^{\gam-1}L(n)$. Consequently, applying \cite[Lemma 5.6]{DM08} with $a$ (there) taken to be $n^\eta$ gives 
\[
\Pn (||Z_s|| \geq n^\eta) \leq 2d \exp\{ - n^{2\eta-\gam}L(n)^{-1}/(2d\ovth) \}
\]
as long as $n^\eta < n \Gamma d = d n^\gam L(n)$, i.e. $\eta < \gam$. To get a nontrivial bound we also need $2\eta > \gam$, so that $\Pn (||Z_s|| \geq n^\eta) \leq n^{-2}$ for $\eta \in (\gam/2,\gam)$ and $n$ large enough. This gives the required conclusion for any 
$\delta > \max\{ \gam/2, 2\gam - 1, \beta +\gam - 1 \}$.
\end{proof}


The process with distribution $\Paux$ has independent increments for $\ta\in J_n$, which allows us to couple it with Brownian motion. Let us fix $J_n = \cup_{i=1}^N J_n^i$, union of finitely many intervals $J_n^i = [\ltai,\uta^i]$ with $\lta^i = \floor{n\thic - n^\beta}$ and $\uta^i = \floor{n\thic + n^\beta}$. We define $\lta^0=\uta^0=0$ for notational convenience. Notice that $|J_n| \leq 2Nn^\beta$.

Define $X^i(t) \equiv \frac12 \Fi t^2 + \sqrt{ \Gi} W^i(t)$, where $\{W^i(\cdot)\}_{i=1}^N$ are independent doubly infinite standard Brownian motions conditioned on $W^i(0) = 0$ and where $\Fi = \frac{\did^2 \y}{\did \th^2}(\thic) = \A(\y(\thic,0)) \F(\y(\thic,0))$ and $\Gi = \G(\y(\thic,0))$. Finally, let
\begin{eqnarray*}
X_n^i(\ta) \equiv n^{1/3} [X^i(n^{-2/3}(\ta - 0.5 - n\thic)) - 
X^i(n^{-2/3}(\ltai - 0.5 - n\thic)) ].
\end{eqnarray*}

\begin{lemma}\label{lem:brownian}
Let $\beta\in (2/3,1)$ and $|\rho| \leq n^{\beta'-1}$ for $\beta' < 2\beta-1$. Then for any  $\delta > 3\beta -2$ there exist finite positive constants $\ka_1$ and $\ka_2$ and a coupling of the process $\zp(\cdot)$ with law $\Paux$ and the Gaussian processes $\lbr X_n^i(\cdot) \rbr_{i=1}^N $, such that for each $i = 1 \ldots N$ we have

\[ \Pnr \lbr \sup_{t \in J^i_n} \ln \, \sfmr_i \cdot\lbr \zp(\floor{t}) - \zp(\ltai) \rbr - X_n^i(t) \rn \geq \ka_1 \,n^\delta \rbr \leq \frac{\ka_2}{n}
\]
\end{lemma}

\begin{proof}
In the following, we assume that $\det\G(\y(\thic)) \neq 0$ for all $i$, as for each $i$ with $\det\G(\y(\thic)) = 0$ the result is just a statement about the approximation of a Riemanian sum by an integral (this is clear from the last paragraph of the proof).

First, notice that the chain $\zp(\cdot)$ has independent increments for $\ta \in J^i_n$ for each $i=1,\ldots, N$ and moreover for $i\neq j$ the increments in $J^i_n$ are independent of those in $J^j_n$. Define random variables $\gamma^i_j \equiv \sfm_i\{ \zp(\ltai+j) - \zp(\ltai+j-1)\}$ for $1 \leq j \leq \uta^i-\ltai$, which are uniformly bounded by $\sfk_1$ and have mean and variance 
\begin{eqnarray}
\En \gamma^i_j = \sfmr_i\cdot\F(\ytn)  \qquad\mbox{and}\qquad  \Var \gamma^i_j = \sfm_i\G(\ytn)\sfmr_i\,.
\end{eqnarray}
We can then apply Sakhanenko's refinement \cite{Sak85} of the Hungarian construction to conclude that there exist independent real-valued Gaussian processes $\{b^i_n(\cdot)\}_{i=1}^N$, with  $b^i_n(\cdot)$ defined on $J^i_n$ and satisfying $b^i_n(\ltai) = 0$, such that
\begin{eqnarray}\label{ineq:sak}
\P_n \lbr \sup_{\ta \in J^i_n} \ln \, \sfmr_i\cdot \lbr \zp(\ta) - \zp(\ltai) \rbr - b_n^i(\ta) \rn \geq c_0 \,\log n \rbr \leq \frac{c_1}{n}.
\end{eqnarray}
The increments $b^i_n(\ta+1)-b^i_n(\ta)$, $\ta\in J^i_n$ have mean and variance equal to the mean and variance of $\gamma^i_{\ta - \ltai}$. The inequality (\ref{ineq:sak}) follows by applying Chebyshev's inequality to the result of \cite[Theorem A]{Sha95}).

Thus for $i = 1 \ldots N$ we have the respresentation
\begin{eqnarray*}
b_n^i(\ta) = \sum_{\si = \ltai}^{\ta-1} \sfmr_i\cdot\F (\ysn) + B^i \ll \sum_{\si = \ltai}^{\ta-1}\sfm_i \G(\ysn)\sfmr_i  \rr
\end{eqnarray*}
where $B^i(\cdot)$ are independent standard Brownian motions. Since the real-valued Gaussian processes $\{X^i_n(t),t\geq\ltai\}$ admit the representation
\begin{eqnarray*}
X^i_n(t) = \Fi\int_{\ltai - 0.5}^{t-0.5}(\si/n - \thic) \,\did\si + B^i(\Gi(t-\ltai/n)),
\end{eqnarray*}
to conclude the proof, it is enough to show that this coupling of $b_n^i(\cdot)$ and $X_n^i(\cdot)$ is such that all $n$ and $i$
\begin{eqnarray}\label{eq:discrete2cts}
\Pn \{ \sup_{t\in J^i_n} | b^i_n(\floor{t}) - X^i_n(t)| \geq 3 c_2 n^\delta \} \leq \frac{c_3}{n},
\end{eqnarray}
where the $t$ in the supremum takes values in $\RR$. Since $\{X^i_n(\ta+s) - X^i_n(\ta) \; : \; 
s \in [0,1]\}$ have the same law as $\{ B^i(\Gi s) + a_{n, \ta}(s) \; : \: s \in [0,1] \}$ for some non-random $a_{n,\ta}(s)$ bounded uniformly in $s\in [0,1]$, $n$ and $\ta \in J^i_n$, we get (\ref{eq:discrete2cts}) as soon as we show that
\begin{eqnarray}
\sup_{\ta\in J^i_n} \Pn \{| b^i_n(\ta) - X^i_n(\ta)| \geq 2 c_2 n^\delta \} \leq n^{-2}.
\end{eqnarray}
This follows by standard Gaussian tail estimates from the fact that $|\En b^i_n(\ta) -\En X^i_n(\ta)| \leq c_4 (n^{2(\beta-1)} + n^{\beta'-1})n^{\beta} \leq c_4 n^{3\beta - 2}$ and $\Var( b^i_n(\ta) - X^i_n(\ta) ) \leq c_5(n^{\beta-1} + n^{\beta'-1})n^{\beta}   = c_6 n^{2\beta -1}$, itself a consequence of  Lemma \ref{lem:figi} and the bounds on $|J_n^i|$ and $|\rho|$. 
\end{proof}


All the above couplings are collected in the following proposition.

\begin{propo}\label{prop:1}
Let $\beta\in (3/4,1)$ and $|\rho| \leq n^{\beta'-1}$ for $\beta' < 2\beta - 1$. Then there exists a positive constant $\ka$, such that for $\delta > (3\beta - 2)\vee\maly$
\begin{eqnarray*}
\Pn \lbr \bKn \geq n^\de \rbr - \frac{\ka}{n} \leq P(n,\rho) \leq 
\Pn \lbr \bKn \geq -n^\de \rbr +\frac{\ka}{n}
\end{eqnarray*}
where $[\bKn]_i\equiv \sfm_i\zp(\ltai) + \inf_{t\in J^i_n} X^i_n(t) - \sfd_i$ and $\zp(\cdot)$ is the process with law $\Paux$.
\end{propo}

\begin{proof}
Taking $r = c_3 \log n$ with some $c_3 > \ka_1/\ka_2$ in Lemma \ref{lem:originalcoupling} we see that there exists a coupling of the Markov chain $\z(\cdot)$ with kernel $W_n(\cdot|\cdot)$  and the chain $\zp'(\cdot)$ with kernel $W(\cdot|n^{-1}\cdot)$ (both started with distribution $\z(0)=\zp'(0)=\zrn$) such that with probability exceeding $1 - 1/n$ the two processes are at most $e_n \equiv c_3 n^{\maly} \log n$ apart until time 
$\tast \equiv \min\{\ta \geq 0:  \z(\ta) \leq 0 \}\wedge\floor{n\ovth}$ and therefore
\begin{eqnarray}
\Pn \lbr \min_{\ta \leq \floor{n\ovth}} \sfd(\zp'(\ta)) \geq e_n \rbr - \frac{1}{n} &\leq&
\Pn \lbr \min_{\ta \leq \tast} \sfd(\oz(\ta)) > 0 \rbr \nonumber\\
&\leq&
\Pn \lbr \min_{\ta \leq \floor{n\ovth}} \sfd(\zp'(\ta) \geq -e_n \rbr + \frac{1}{n} \label{eq:fromorigcoupling}
\end{eqnarray}

Next, observe that by the bound of Lemma \ref{prop:ode} and the assumption (\ref{ass:mean}) the events $\{\min_{\ta \in I_n}\sfd(\zp'(\ta)) \leq \pm e_n\}$ have probability smaller than $c_4/n$ (taking $c_3$ big enough) and hence we can restrict our attention to $\min_{\ta \in J_n}\sfd(\zp'(\ta)) $. Transfering the couplings in Proposition \ref{prop:general} (whose hypothesis is satisfied for the $J_n$ we are considering here with $\gam=1/2$ and $L(n)=\log n$) and Lemma \ref{lem:brownian} onto one probability space, it follows that
\begin{eqnarray*}
&&\Pn \lbr \sfm_i\zp(\ltai) + \inf_{t\in J^i_n} X^i_n(t) - \sfd_i \geq n^{\de},\; i\leq N \rbr - \frac{c_5}{n} \\
&\leq& \Pn \lbr \min_{\ta \in J_n} \sfd(\zp'(\ta)) \geq \pm e_n \rbr \\
&\leq& \Pn \lbr\sfm_i\zp(\ltai) + \inf_{t\in J^i_n} X^i_n(t) - \sfd_i \geq -n^{\de},\; i \leq N\rbr + \frac{c_5}{n} 
\end{eqnarray*}
where $\de > (3\beta - 2)\vee\maly$. Combining this with the estimate (\ref{eq:fromorigcoupling}) yields the conclusion of the proposition.
\end{proof}

\section{Weak convergence arguments}\label{sec:weak}
Having reduced our problem to that of estimating the probability that the random vector $\bKn$ lies in a convex set, we can now leverage a convenient decomposition of $\bKn$ into a sum of independent vectors, that will enable us to apply the central limit theorem and subsequently use the Taylor expansion for a smooth probability distribution function in $\RR^N$. Our methods here depart from those of \cite{DM08} and seem to simplify the argument leading to a somewhat tighter bound on the error.

To begin with, let us define the $\RR^N$ vectors with entries ($1\leq i \leq N$):
\begin{eqnarray}
[\bGn]_i &\equiv& \sfm_i \lbr\yst(\ltai) + \Enr \z(0) - n\y(0,\rho)\rbr + \frac{\Fi}{2}n^{1/3} s_n^2 - \sfd_i 
\end{eqnarray}
as well as symmetric matrices $\bSn$ whose entries are given by ($1\leq i\leq j \leq N$) 
\begin{eqnarray}
\relax [\bSn]_{i j} &\equiv& \sfm_i \,\Qst(\ltai)\, \tB(i,j)^\dag \, \sfmr_j 
\end{eqnarray}
where $\tB(i,j) = \tB_{\utai}^{\lta^j}$. In what follows, also $\B(i,j) = \tB_{\thic}(\thjc)$. 

\begin{lemma} \label{lem:sigmagamma}
Let $r\in \RR$ and take $\rho = \rho_n \equiv rn^{-1/2}$ in the definition of $\bSn$ and $\bGn$. Then there is a positive constant $\ka$ such that for all $\beta \in (3/4, 1)$ and for all $n$
\begin{eqnarray}
|| n^{-1}\bSn - \bS || &\leq& \ka n^{\beta -1} \label{var} \\
|| n^{-1/2}\bGn - r\bG || &\leq& \ka n^{3\beta - 5/2}  \label{mean}
\end{eqnarray}
\end{lemma}
\begin{proof}
We suppress the dependence of $\Q$ on $\rho$ and start by expanding the difference
\begin{eqnarray*}
 n^{-1}\Qst(\ltai)\tB(i,j)^\dag-\Q(\thic)\B(i,j)^\dag 
= n^{-1}\Qst(\ltai) \lsq \tB(i,j)-\B(i,j)\rsq^\dag + \lsq n^{-1}\Qst(\ltai)-\Q(\thic)\rsq \B(i,j)^\dag
\end{eqnarray*}
to see that (\ref{var}) follows from the estimates proved in Lemma \ref{lem:meanvar} and the uniform bound on the norms of all matrices involved. In turn to see (\ref{mean}), note that
the proof of Lemma \ref{lem:figi} gives the bound
\[
\ln \sfm_i\yst(\floor{n\thic}) - \sfm_i\yst(\ltai) + 
\Fi\sum_{\si = \ltai}^{\floor{n\thic}}(\si/n - \thic) \rn \leq c_0 n^{3\beta - 2},
\]
a simple calculation confirms that
\[
 \ln \frac{1}{2}n^{1/3} s_n^2 -  \sum_{\si = \ltai}^{\floor{n\thic}}(\si/n - \thic) \rn \leq \frac{1}{8}n^{-1},
\]
while it is shown in Lemma \ref{lem:meanvar} that
\[
\l| n^{-1}\yst(\floor{\thic n}) - \y(\thic,\rho) \r| \leq c_1 n^{-1},
\] 
and from Lemma \ref{lem:solutionproperties} we have
\begin{eqnarray*}
\ln \y(\thic,\rho) - \y(\thic,0) - \rho \frac{\pip \y}{\pip \rho}(\thic,0) \rn \leq c_2 n^{\beta - 1}|\rho|.
\end{eqnarray*}
Finally $|\Enr\z(0) - n\y(0,\rho)| \leq \sfk_5$ by (\ref{ass:mean}). Putting these estimates together, and using $\rho_n = r n^{-1/2}$ and $\sfd_i = \y(\thic,0)$, completes the proof, since for $\beta \in (3/4,1)$ we have $-1/2 < \beta - 1 < 3\beta - 5/2$.
\end{proof}

The next lemma is a bound on the distance between Gaussian distributions in terms of the differences between their means and covariances.

\begin{lemma} \label{lem:gaussianbound}
Let $\xx$ and $\yy$ be Gaussian vectors in $\RR^N$ with means $\aa$, $\bb$ and positive definite covariance matrice $\AA$, $\BB$ respectively. Let $\de > 0$ be a lower bound for the eigenvalues of both $\AA$ and $\BB$. Then there exists a constant $\ka$ depending only on $N$ and $\de$ such that
\begin{eqnarray*}
\sup_{C \in \cC} \ln \P ( \xx  \in C ) - \P ( \yy \in C ) \rn 
\leq \ka \lbr ||\aa - \bb|| + ||\AA - \BB|| \sqrt{\logp(||\AA - \BB||)} \rbr,
\end{eqnarray*}
where $\cC$ denotes the collection of Borel-measurable convex sets in $\RR^N$ and $\logp x \equiv |\log x|$ for $x>0$ and equals 0 otherwise.
\end{lemma}
\begin{proof} We couple $\xx$ and $\yy$ by taking $\xx = \SS \zz$ and $\yy = \TT \zz$,
where $\zz$ a standard normal Gaussian vector in $\RR^N$ and $\SS$ and $\TT$ are symmetric positive definite square-roots of $\AA$ and $\BB$ respectively, i.e. $\SS\SS^\dag = \AA$ and
$\TT\TT^\dag = \BB$. Let $\al = ||\SS - \TT||$ and let $r = \al \sqrt{3 \logp\al} + ||\aa - \bb||$. Observe that for $\al > 0$ we have
\begin{eqnarray}\label{eq:tail1}
 \P( ||\xx - \yy|| > r) \leq \P( ||\SS - \TT|| ||\zz|| + ||\aa - \bb|| > r)
\leq \P (||\zz|| > \sqrt{3 \logp\al}) \leq \ka_1 \al 
\end{eqnarray}
for some $\ka_1 \geq 0$, where we have used the standard Gaussian tail estimate $\P (||\zz|| \geq u) \leq p_{N-2}(u) e^{-u^2/2}$ for some polynomial $p_{N-2}$ of degree $(N-2)\vee 0$ with coefficients depending only on $N$. In turn for $\al = 0$ we have $r = ||\aa - \bb||$, so the probability in (\ref{eq:tail1}) is trivially 0. Next, conditioning on $||\xx - \yy|| \leq r$ and using monotonicity of measure gives
\begin{eqnarray}
\P(\xx \in C) \leq \P(\yy \in C^r) +\ka_1 \al \leq 
\P(\yy \in C) + \ka_2 r + \ka_1 \al
\end{eqnarray}
where $C^r \equiv \{x \in \RR^N: \dist(x, C) \leq r\}$ and the last inequality for some $\ka_2>0$ depending only on $N$ follows by a uniform bound on integrals over convex shells of width $r$ (see for example \cite[Corollary 3.2]{BR76}). Since by symmetry the above inequality also holds with $\xx$ and $\yy$ swapped, to complete the proof we only need to show that $||\SS - \TT||  \leq \ka_3 || \AA - \BB ||$ for some positive $\ka_3$ that may depend on $N$ and $\de$. As all norms in $\RR^{N^2}$ are equivalent, we will demonstrate this in Frobenius norm, defined by $||\AA||_\Fr \equiv \Tr (\AA^2)$. 

To this end, we write $(\SS-\TT)^2 = O L O^\dag$ and $(\SS+\TT)^2 = U D U^\dag$, where $O$, $U$ are orthogonal and $L$, $D$ are diagonal and positive. This is possible because $\SS$ and $\TT$ are positive definite and symmetric. Let the elements on the diagonal of $L$ be $l_i$ and on the diagonal of $D$ be $d_i$, where $d_i > 2\delta$, which follows from $(\SS+\TT)^2 = \AA + \BB + 2\SS\TT$, where all the matrices are positive definite and $\de$ is a positive lower bound on the eigenvalues of $\AA$ and $\BB$. Thus
\begin{eqnarray*}
\Tr[(\AA -\BB)^2] &=& \Tr[(\SS-\TT)^2(\SS+\TT)^2] = \Tr[OLO^\dag UDU^\dag]\\ 
&=&  \Tr[UOLO^\dag UD] = \sum d_il_i > 2\de \Tr[(\SS-\TT)^2]
\end{eqnarray*}
demonstrating that $||\SS -\TT||_\Fr \leq (2\de)^{-1}||\AA -\BB||_\Fr$ and concluding the proof.
\end{proof}

In preparation for the next proposition, we introduce a decomposition of $\bKn$ into a linear combination of random vectors $\bYn$, $\bZn$, $\bHn$, and $\bXn$ with entries ($1\leq i \leq N$)
\begin{eqnarray}
[\bYn]_i &\equiv& \sum_{0 \leq \si \leq \ltai, \si \notin J_n}\sfm_i \tB_{\si+1}^{\ltai} \lbr\D_\si - \F(\ysn)\rbr  \, \\
\relax [\bZn]_i &\equiv& \sfm_i\tB(0,i) (\z(0) - \Enr\z(0)) \\
\relax [\bHn]_i &\equiv& n^{-\beta/2}\sum_{\si \in J_n, \si \leq \ltai,}\sfm_i \tB_{\si+1}^{\ltai} \lbr\D_\si - \F(\ysn)\rbr \\
\relax [\bXn]_i &\equiv& \inf_{t \in [-s_n,u_n]}X^i(t) - X^i(-s_n) + \frac12 \Fi s_n^2,
\end{eqnarray}
where $s^i_n = n^{-2/3}[n^\beta+0.5]$ and $u_n = n^{-2/3}[n^\beta - 0.5]$, so that
\begin{eqnarray*}
\sfmr_i\cdot \zp(\ltai) = \sfmr_i\cdot\yst(\ltai) + [\bYn]_i + [\bZn]_i + n^{\beta/2}[\bHn]_i.
\end{eqnarray*}
and consequently
\begin{eqnarray}
\bKn = \bGn + \bYn + \bZn + n^{\beta/2}\bHn + n^{1/3}\bXn.
\end{eqnarray}
The merit of writing $\bKn$ like this lies in the fact that $\bYn, \bZn$, and $(\bHn, \bXn)$ are mutually independent and centered.

\begin{propo} \label{prop:2}
Suppose that $\bS$ is positive definite and fix $\beta_0 \in (3/4,1)$ and
$\de > (3\beta - 2)\vee\maly$. Then there exists a positive $\ka$, such that for any  $\beta \in (3/4,\beta_0)$, $\eta < \min\{1/2-\delta,1-\beta,3 \beta - 5/2\}$, $r \in \RR$, $\rho_n = rn^{-1/2}$ and $n$ sufficiently large we have
\begin{eqnarray*}
\ln \Pn\lbr  \bKn \leq \pm n^{\delta} \rbr - \Pn \lbr \bY + n^{-(1-\beta)/2} \bHn + n^{-1/6} \bXn \leq \bzer \rbr
\rn \leq \ka n^{-\eta},
\end{eqnarray*}
where  $\bY$ is an independent Gaussian vector in $\RR^N$ with mean $r \bG$ and covariance matrix $\bS$.
\end{propo}

\begin{proof} We rescale $\bKn$ by a factor of $n^{-1/2}$ and observe that for any random vector $y$, independent from the random vector pair $(x_1,x_2)$, we have
\begin{eqnarray*}
\ln \P (x_1 \leq y) - \P(x_2 \leq y) \rn &\leq& \E \ln\P(x_1\leq y | y) - \P(\x_2\leq y | y)\rn\nonumber\\
&\leq& \sup_{z \in \RR^N} \ln\P(x_1\leq z) - \P(x_2\leq z)\rn.
\end{eqnarray*}
Thus conditioning on $(\bHn,\bXn)$ (which is independent from $\bYn$ and $\bZn$) we will be done once we show that
\begin{eqnarray}
\sup_{\xx \in \RN} \ln \Pn\lbr  n^{-1/2}(\bGn + \bYn + \bZn)  \pm n^{\delta-1/2} \leq \xx  \rbr -
\Pn \lbr  \bY \leq  \xx   \rbr \rn \leq n^{-\eta}. \label{a1}
\end{eqnarray}
To this end first recall that by our assumption (\ref{neargaussian}) we have 
\begin{eqnarray*}
\sup_{\U \in \cM_{N,d}}\,\sup_{\xx \in \RR^d} 
\ln \Pnr\lbr  n^{-1/2}\U \lbr\z(0)-\Enr\z(0)\rbr \leq \xx\rbr -
\P \lbr n^{1/2}\U\vze_\rho \leq  \xx   \rbr \rn \leq \sfk_6 n^{-1/2},
\end{eqnarray*}
where $\vze_\rho$ is a centered Gaussian vector in $\RR^d$ with positive semi-definite covariance matrix $\Q_\rho$. Setting $\U = \U_n \equiv [ \sfmr_1\tB(0,1), \ldots, \sfmr_N\tB(0,N)]^\dag$, writing $\bAn = n^{1/2}\U\vze_\rho$ and conditioning on $\bYn$ (indepedent from $\bZn$) in (\ref{a1}) immediately reduces our task to that of showing
\begin{eqnarray*}
\sup_{\xx \in \RN} \ln \Pnr\lbr  n^{-1/2}( \bYn + \bAn) +  n^{-1/2} \bGn \pm n^{\delta-1/2} \leq \xx  \rbr -
\P \lbr  \bY \leq  \xx   \rbr \rn \leq n^{-\eta}.\label{a2}
\end{eqnarray*}
This  follow will from Lemmas \ref{lem:sigmagamma} and \ref{lem:gaussianbound}
once we prove that for $\bN$ a standard normal vector in $\RR^N$ we have
\begin{eqnarray}\label{a3}
\sup_{\xx \in \RN} \ln \Pnr\lbr  n^{-1/2}( \bYn + \bAn) \leq \xx \rbr -
\Pn \lbr  \bVn \bN \leq \xx  \rbr \rn \leq n^{-\eta},
\end{eqnarray}
where $\Var(\bYn +\bAn) = \bSn = \bVn\bVn^\dag$.

Since $\bS$ is assumed to be positive definite, by (\ref{var}) for $n$ large enough $\bS_n$ is positive definite, too. Thus (\ref{a2}) is a consequence of
\begin{eqnarray}\label{eq:CLT}
\sup_{C \in \cC} \ln \Pnr\lbr  n^{-1/2}\bVn^{-1}( \bYn +\bAn)  \in C \rbr -
\P \lbr \bN \in C   \rbr \rn \leq n^{-1/2},
\end{eqnarray}
where as before $\cC$ denotes the collection of Borel-measurable convex sets in $\RN$ and $\bN$ is a standard normal vector in $\RR^N$. Indeed it is enough to consider the convex sets $C = C(\xx) \equiv \{\xx'\in\RN : \bVn \xx' \leq \xx \}$.

Thus we conclude by proving (\ref{eq:CLT}) with an application of \cite[Theorem 13.3]{BR76}. Namely observe that $\bYn$ is a sum of $n$ independent centered random vectors $\bg_k$ in $\RN$  where
\begin{eqnarray*}
[\bg_k]_i = \1_{k \leq \ltai} \1_{k \in I_n} \sfm_i \tB_k^{\ltai} \lbr\D_k-\F(\y(k/n))\rbr
\end{eqnarray*}
and these vectors have uniformly (in $n$ and $\rho$) bounded fourth moments ($|\D_\ta| \leq \sfk_1$). Similarly the self-decomposability of the normal distribution gives us i.i.d. centered Gaussian vectors $\widehat{\bg}_k$ such that $\bAn = \sum_{k=1}^{n}\widehat{\bg}_k$. Hence applying \cite[Theorem 13.3]{BR76} to the independent vector array $\bVn^{-1}(\bg_k+\widehat{\bg}_k)$, whose sum (over $k$ for each $n$ large enough) has a covariance matrix $n\ind_N$, yields the required result.
\end{proof}

\subsection{Proof of Theorem \ref{thm:markovexit}}

Combining the results of Propositions \ref{prop:1} and \ref{prop:2} we get
\begin{eqnarray*}
P_\nex(n,\rho_n) &=& \Pn \lbr \bG + \bY + n^{-(1-\beta)/2}\bHn + n^{-1/6}\bXn \geq 0 \rbr + O(n^{-\eta}) \\
  &=& \En \lbr \Phi\lsq \bS^{-1/2}\ll \bG + n^{(\beta-1)/2}[ \bHn + \bXn^{(1)}]+ n^{-1/6}\bXn^{(2)}\rr\rsq  \rbr + O(n^{-\eta}).
\end{eqnarray*}
where $[\bXn^{(1)}]_i \equiv n^{\beta/2 - 1/3} (X^i(-s_n) - \Fi s_n^2/2) $ and  $[\bXn^{(2)}]_i \equiv  \inf_{t\in [-s_n,u_n]}X^i(t)$. Our normalization and centering ensures that (uniformly in $n$)
\begin{eqnarray*}
\En [\bHn]_i = 0\;, \qquad \Var[\bHn]_i \leq C\;, \qquad
\En [\bXn^{(1)}]_i = 0\;, \qquad\mbox{and}\qquad \Var[\bXn^{(1)}]_i \leq C.
\end{eqnarray*}
We will now show that there also exist positive constants $c_1, c_2$ such that for all $1\leq i\leq N$ we have
\begin{eqnarray*}
 \E \ln \inf_{t\in [-s_n,u_n]}X^i(t) + \Fi^{-1/3}\Gi^{2/3}\Omega \rn &\leq& c_1\Fi^{-1/3}\Gi^{2/3} n^{-1} \\
 \E \ln  \inf_{t\in [-s_n,u_n]}X^i(t) \rn^2 &\leq & c_2 \Fi^{-2/3}\Gi^{4/3}\E |V|^2 < \infty
\end{eqnarray*}
where  $V \equiv  \inf_{t\in \RR}X(t)$ for $X(t) \equiv \frac12t^2 + W(t)$ with $\proc{W}$ a double-sided standard Brownian motion. 

In \cite[Theorem 3.1]{Gro89} it is shown that the distribution function of $V$ is $F_V(v) = 1 - \cK(-v)^2\1_{v<0}$, where $\cK$ is defined in the (\ref{def:cK}) and this explicit formula follows from \cite[(5.2)]{Gro89} taken with $c=1/2$ and $s=0$. Integration by parts reveals that 
$\E V = -\Omega$, for the $\Omega$ defined in (\ref{def:Omega}).

As a consequence of the above and the Brownian scaling giving 
\[
X^i(t) \stackrel{\cal L}{=} \Fi^{-1/3}\Gi^{2/3} X(\Fi^{2/3}\Gi^{-1/3}t)
\]
(the case $\Gi=0$ is trivial), it is enough to show that
\begin{eqnarray}
 \E \ln V_n - V \rn &\leq& c_1n^{-1} \label{eq:1st}\\
 \E \ln  V_n \rn^2 &\leq & c_2 \E |V|^2 < \infty \label{eq:2nd},
\end{eqnarray}
with  $V_n \equiv  \inf_{t\in [-S_n,U_n]}X(t)$, where $X(\cdot)$ is the same doubly infinite Brownian motion as in the definition of $V$ and $S_n$ and $U_n$ are positive sequences increasing to infinity, such that $T_n\equiv S_n\wedge U_n \geq \Fi^{2/3}\Gi^{-1/3}n^{\beta-2/3}$. To this end notice that (\ref{eq:2nd}) follows immediately from $0 \geq V_n \geq V$. To see (\ref{eq:1st}), define the random time  
\[
T \equiv \inf \{s>0 : X(s) = \inf_{t\in \RR}X(t) \;\mbox{ or }\; X(-s) = \inf_{t\in \RR}X(t) \}
\]
and apply the Cauchy-Schwarz inequality to obtain
\begin{eqnarray*}
\E\ln V_n - V \rn \leq \E \lbr \ln V_n - V \rn \1_{|T| > T_n} \rbr \leq
 \l|V_n - V \r|_2 \P\lbr |T| > T_n \rbr^{1/2} \leq  \l| V \r|_2 \P\lbr |T| > T_n \rbr^{1/2}.
\end{eqnarray*}
In \cite[Corollary 3.4]{Gro89} it is proved that $\P\lbr |T| > t \rbr \leq A_0^{-1}e^{-A_0 t^3}$, which together with $T_n > \Fi^{2/3}\Gi^{-1/3}n^{\beta-2/3}$ will imply (\ref{eq:1st}) once we show that $\E |V|^2 < \infty$. This is achieved by using the same estimate on $\P\lbr |T| > t \rbr$  and bounding as follows
\begin{eqnarray*}
\P(V \leq v) \leq \P \lbr \inf_{s \in [-t,t]}X(s) \leq v \rbr + A_0^{-1}e^{-A_0 t^3} 
\leq \lbr \inf_{s \in [-t,t]}W(s) \leq v \rbr + A_0^{-1}e^{-A_0 t^3} 
\leq  e^{-v^2/2t} + A_0^{-1}e^{-A_0 t^3},
\end{eqnarray*}
so that taking $t = \sqrt{v}$ we can deduce that $\P(V \leq v) < C^{-1}\exp(-C |v|^{3/2})$ for some $C>0$ and all $v<0$. In fact, it follows that all moments of $V$ are finite.

To conclude, we use second order Taylor expansion for $\Phi$ and the bound $||\bS^{-1}|| \leq c_3$ for all $n$ to obtain
\begin{eqnarray*}
P(n,\rho_n) = \Phi \ll r \bS^{-1/2}\bG\rr - 
\ll\bS^{-1/2}\bLam \rr^\dag \nabla\Phi \ll r \bS^{-1/2}\bG \rr \Omega n^{-1/6} + 
c_4 n^{\max\{(\beta-1),-1/3,-1/6-1,-\eta\}}.
\end{eqnarray*}
where $[\bLam]_i \equiv \Fi^{-1/3}\Gi^{2/3}$.
Finally, observe that since for $\beta \in (3/4,1)$ we have $\eta < 5/2 - 3\beta < 1-\beta < 1/4$ with $5/2 - 3/4 = 1/4$, the above holds for all $\eta < 1/4$.
\endproof

\section{Finite-size scaling for the 2-core in irregular hypergraphs}\label{sec:hypergraph}

This section is devoted to the analysis of the phase transition for the existence of the 2-cores in hypergraph ensembles and specifically the proof of Theorem \ref{thm:hypergraph}. Section \ref{sec:hypermodel} defines the state-space paramaterizing ensembles of interest. The peeling algorithm and the density dependent Markov chain deriving from it is treated in Section \ref{sec:hyperexact}. The asymptotic probability kernel $W(\cdot|\cdot)$ is furnished in Section \ref{sec:hyperasymptotic}. In Section \ref{sec:ensemble1} it is shown that the paramtetrization of the initial ensemble satifies the hypothesis (\ref{ass:exp})-(\ref{neargaussian}). Finally, the proof of Theorem \ref{thm:hypergraph} is in Section \ref{1stproof}.

\subsection{State-space of ensembles}\label{sec:hypermodel}
We model hypergraphs as bipartite graphs with the set of vertices partitioned into v-nodes and c-nodes (the \emph{configurational} model). The hyperedges correspond to v-nodes and hypervertices to c-nodes, and the resulting hypergraphs is more precisely a hyper-multigraph.

Fix integers $K \geq 2$ and $L \geq 3$. Every ensemble in the state-space is characterized by the non-negative integers 
\[ \oz = (z_1,\ldots,z_{K+L-2}) \equiv (\vo, \t) = (\om_1,\om_2,\ldots,\om_K,\ta_3,\ldots,\ta_L)\;, \qquad n, \qquad \mbox{and} \qquad m,
\]
and denoted $\cG(\oz) = \cG(\vo,\t)$. We define $\dd(\t) \equiv \sum_{j=3}^Lj\ta_j$ and $\hta(\t) \equiv n - \sum_{j=3}^L\ta_j$, where we drop the explicit argument, when it is clear from the context. 

In order for $\cG(\oz)$ to be non-empty, we require that $\sum_{i=1}^Kz_i \leq m$, $\sum_{j=3}^L\ta_j \leq n$ and either $z_K \geq 1$ and $\sum_{i=1}^Kiz_i \leq \dd(\t)$, or $z_K = 0$ and $\sum_{i=1}^Kiz_i = \dd(\t)$. An element in the ensemble is a bipartite graph 
\[ G = (U,V_3,V_4,\ldots,V_L; R_0, R_1,\ldots, R_{K}; E),
\] 
where (denoting disjoint union by $\amalg$) we have $U\amalg V_3 \amalg \cdots \amalg V_L = [n]$ (called the set of v-nodes) and $R_0\amalg \cdots \amalg R_{K} = [m]$ (called the set of c-nodes), and the cardinalities of these sets are $|U| = \hta$, $|V_j| = \ta_j$, $|R_0| = m - \sum_{i=1}^Kz_i$ and $|R_i| = z_i$. Finally $E$ is an ordered list of $n-\hta$ edges
\begin{eqnarray*}
 E = [ (\al_1,a_1),(\al_1,a_2),\ldots,(\al_1,a_{l_1}); (\al_2,a_{l_1+1}),(\al_2,a_{l_1+2}),\ldots,(\al_2,a_{l_1+l_2}); \ldots;\\
(\al_{n-\hta},a_{l_1+\cdots + l_{n-\hta-1} + 1}),\ldots,(\al_{n-\hta},a_{l_1+\cdots + l_{n-\hta}})  ] 
\end{eqnarray*}
such that the pair $(\al,a)$ appears before $(\beta, b)$ whenever $\al < \beta$. Moreover, each $\al \in V_j$ appears in the list exactly $j$ times, while none of the $\al \in U$ appear in any of the pairs in $E$. Analogously, none of the $a \in R_0$ appears in $E$; each $a \in R_i$ for $1\leq i \leq K-1$ appears in exactly $i$ edges; however, each $a \in R_K$ appears in \emph{at least} $K$ edges. The total number of graphs in $\cG(\oz)$ is thus
\begin{eqnarray}
h(\oz)  \equiv |\cG(\oz)| = 
\binom{m}{z_1,\ldots, z_K,\cdot} \binom{n}{\ta_3,\ldots,\ta_L,\cdot}
\lbr d! \prod_{i=0}^K(i!)^{-z_i}\rbr \coeff[e_K(\sfx)^{z_K},\sfx^{\ovd}]
\end{eqnarray}
where $\ovd = \ovd(\oz) \equiv \dd(\t) - \sum_{i=1}^{K-1} iz_i$ and 
$e_k(x) \equiv \sum_{i=k}^\infty x^{i}/i!$.

\subsection{Exact kernel}\label{sec:hyperexact}
We consider the following graph-valued process $\proc{G}$. We assume that the distribution of the initial graph $G(0)$ is such that conditioned on $\{ G(0) \in \cG(\oz) \}$ (when this event has positive probability), it is a uniformly random element of $\cG(\oz)$. At each time $\si = 0, 1, \ldots$ if the set of c-nodes of degree 1 is non-empty, one of them is chosen uniformly at random. Let this c-node be $a$ and observe that there is a unique v-node $\al$, such that $(\al,a) \in E$. This edge and all other edges $(\al, \cdot) \in E$ (edges incident to the v-node $\al$) are deleted and the graph thus obtained is denote by $G(\si+1)$. Otherwise, that is if there are no c-nodes of degree 1 in $G(\si)$, we set $G(\si+1) = G(\si)$.

We can now define the process $\{\oz(\si) = (\vo(\si),\t(\si)), \si \geq 0 \}$ on $\ints_+^{K+L-2}$, where $z_i(\si) = \om_i(\si)$ is the number of c-nodes of degree $i$ in $G(\si)$ ($1 \leq i \leq K$) and  $z_{K+j-2}(\si) = \ta_j(\si)$ is the number of v-nodes of degree $j$ in $G(\si)$ ($3\leq j \leq L$). Notice that we have
\[
\sum_{i=1}^Kiz_i(\si) \leq d(\t(\si)) = \sum_{j=3}^Lj\ta_j(\si),
\]
and that $\hta = \hta(\t(\si)) = \si \wedge \min\{\si'\geq 0 : \z_1(\si') = 0 \}$.

\begin{lemma}\label{lem:hyperexact}
The process $\oz(\cdot)$ is a Markov chain with kernel $W_n$ defined by: 
if $z_1 = 0$, then $W_n(\D\oz|\oz) = \1_{\D\oz = 0}$; if $z_1 > 0$,  then  $W_n(\D\oz|\oz) = 0$
unless there is exists a unique $\ell\in\{3,\ldots,L\}$ such that $\D\ta_i = -\1_{i=\ell}$ 
(then in particular $\D d \equiv d(\D\t) = -\ell$) and we have
\begin{eqnarray}
W_n(\D\oz|\oz) = \frac{h(\zp)}{h(\z)}N(\zp | \z), \label{eq:W}
\end{eqnarray}
where $\zp = \z + \D\z$ and
\begin{eqnarray}
N(\zp | \z) \equiv (\ta + 1) \ell! \sum_{\cD}
\binom{m - \sum_{i=1}^Kz_i'}{q_{01},\ldots,q_{0K}} \binom{z_1'}{q_{12},\ldots,q_{1K}}\cdots \binom{z_{K-1}'}{q_{(K-1)K}}\binom{z_K'}{q_{KK}} q_{01} \;\pi (Q)  \label{eq:N}
\end{eqnarray}
with $\ta = n - \hta = n - \sum_{j=3}^L\ta_j$ and
\begin{eqnarray}
\pi(Q) \equiv \coeff\lsq e_1(\sfx)^{q_{KK}}\prod_{j=0}^{K-1}e_{K-j}(\sfx)^{q_{j K}} , \sfx^{\ovell} \rsq,
\end{eqnarray}
where $\ovell \equiv  \sum_{0\leq i < j < K}(j-i)q_{ij}$
and the collection $\cD$ consist of all triangular arrays of non-negative integers $Q \equiv \{q_{ij} \geq 0$: $0\leq i \leq j \leq K\}$, such that $q_{ii} = 0$ for $1\leq i\leq K-1$ and the remaining entries satisfy the system (writing $z_0 \equiv m - \sum_{i = 1}^Kz_i$ and analogously for $z_0'$)
\begin{equation}
\lbr  \begin{array}{rclclcl}
z_0 &=& z_0' &-& \sum_{j=1}^Kq_{0j} \\
z_1 &=& z_1' &-& \sum_{j=2}^Kq_{1j} &+& q_{01}\\
&\vdots& \\
z_i &=& z_i' &-& \sum_{j=i+1}^Kq_{ij} &+& \sum_{j=0}^{i-1}q_{ji} \\
&\vdots& \\
z_{K-1} &=& z_{K-1}' &-& q_{(K-1) K}&+& \sum_{j=0}^{K-2}q_{j (K-1)} \\
z_K &=& z_K' && &+& \sum_{i=0}^{K-1}q_{jK} 
\end{array} \right.
\end{equation}
and the inequalities 
\begin{equation}
\lbr  \begin{array}{rcl}
d(\ta) -  \sum_{i=1}^Kiz_i &\geq& \ell - \sum_{0\leq i < j \leq K}(j-i)q_{ij}\\
\ell - \sum_{0\leq i < j \leq K}(j-i)q_{ij} &\geq& q_{KK} \\
z'_i &\geq& \sum_{j=i}^Kq_{ij}\qquad\mbox{for}\;\;1\leq i \leq K
\end{array} \right..
\end{equation}

Moreover, conditional on $\{\oz(\si'): 0\leq\si'\leq\si \}$, the graph $G(\si)$ is distributed uniformly over $\cG(\oz) \equiv \cG(\oz(\hta))$, i.e.
\begin{eqnarray}
\P \{ G(\si) = G | \{\oz(\si'): 0\leq\si'\leq\si \} \} = 
\frac{1}{h(\oz(\hta))} \1_{G \in \cG(\oz(\hta))} \label{eq:uniform}
\end{eqnarray}
\end{lemma}

\begin{proof}
Fix $\oz=(\vo,\t)$ with $z_1>0$ together with $\ozp = (\vop,\tp)$ and $G' \in \cG(\ozp)$ such that transition happens with positive probability. Let $N(G'|\oz)$ be the number of pairs of graphs
$G\in \cG(\oz)$ and choices of the deleted c-node from $R_1$ that would result in our algorithm producing $G'$. Clearly, the following relations need to be satisfied (primed letters correspond to the analogous sets in $G'$)
\[
R_i \subseteq \bigcup_{j=0}^{i} R'_j\quad\mbox{for}\;\;0\leq i\leq K,\qquad R_K'\subseteq R_K,
\]
and for $0\leq i < j \leq K$ we let $q_{ij} \equiv |R'_i\cap R_j|$. In turn, let $R^\ast_K \subseteq R'_K\cap R_K$ be such that each $a \in R^\ast_K$ had its degree decreased during the algorithm step (i.e. $k_a > k'_a$). It follows that
\[
\ovd \geq \ovell \geq q_{K}, \qquad \sum_{j=i+1}^Kq_{0j} \leq z_i', \qquad\mbox{and}\qquad q_{KK} \leq z_K'
\]
where the first inequality is a consequence of the fact that upon the deletion of a degree $\ell$ v-node at least $(j-i)$ edges of c-nodes in $R'_i\cap R_j$ disappear compared to at least one edge disappearing from all the c-nodes in $R_K$.

We proceed to compute $N(G'|\oz)$. First, select a v-node $\al$ to add to $G'$ among the $\ta + 1 = n - d'$ elements of $U$; by our choice of $\ozp$, the degree of $\al$ is $\ell$ and we select a permutation of its $\ell$ sockets that will be used to connect to the c-nodes and create $G \in \cG(\oz)$. 

Second, we sum over the set $\cD$ of all the possible arrays $Q$. For each possible pair in this set there are $\binom{m - \sum_{i=1}^Kz_i'}{q_{01},\ldots,q_{0K}}$ ways of selecting nodes in $R'_0$ to be assigned to $R_1,\ldots,R_K$; for each $1\leq i \leq K-1$ there are exactly $\binom{z_i'}{q_{i(i+1)},\ldots,q_{iK}}$ ways of selecting nodes in $R'_i$ to be assigned to $R_{i+1},\ldots,R_K$; while there are $\binom{z_K'}{q_{K}}$ ways of selecting those c-nodes in $R'_K$ to be assigned to $R^{\ast}_K$. We have thus allocated $\sum_{0 \leq i < j < K}(j-i)q_{ij}$ edges among the $\ell$ emanating out of $\al$.

Third, we need to select the precise number ($\geq K-j$) of edges from our v-node $\al$ that we will connect to each of the $q_{jK}$ c-nodes in $R'_j\cap R_K$; at the same time we select the number of edges to be added to each of the $q_{KK}$ c-nodes in $R^{\ast}$. Since we allocate in this way exactly $\ovell = \ell - \sum_{0 \leq i < j < K}(j-i)q_{ij}$ edges (emanating out of $\al$), this can be accomplished in exactly $\pi(Q)$ ways.
%
%

Fourth, recall that we are counting not as much graphs $G$ as pairs composed of a graph $G$ and  a particular c-node from $R'_0\cap R_1$ of that graph, the choice of which during the step of our algorithm would produce $G'$. There are $q_{01}$ such c-nodes. We have now proved that $N(G'|\ozp)$ is indeed given by the formula (\ref{eq:N}).

We are now ready to conclude the proof of the theorem by demonstrating (\ref{eq:W}) and (\ref{eq:uniform}). Notice that $N(G'|\ozp)$ depends on $G'$ only through $\ozp$ and recall that the graph-valued chain starts at some distribution $G(0)$ which is uniform conditioned on each ensemble $\cG(\oz(0))$. Consequently, by induction on $\si$, as long as $z_1(\si) > 0$, we have
\begin{eqnarray*}
\Pn \{ G(\si+1) = G' | \{\oz(\si'): 0\leq\si'\leq\si \} \} = 
\frac{N(G'|\oz(\si))}{z_1(\si)\;h(\oz(\si))} 
\end{eqnarray*}
and this is the same for all $G' \in \cG(\ozp)$. (Factor of $z_1(\si)$ comes from us counting the pairs: graph and distinguished c-node of degree 1.)  Thus (\ref{eq:W}) is implied by the cardinality of that ensemble being $h(\ozp)$. Finally, with the graph $G(\si)$ being constant for all $\si$ with $\si > \hta$, these arguments have shown that the property (\ref{eq:uniform}) holds for all $\si$.
\end{proof}

\subsection{Smooth asymptotic kernel}\label{sec:hyperasymptotic}

We next demonstrate that there indeed exists a probability kernel $W$ that approximates the kernel $W_n$ on the domain $\dom(\eps)$, defined below. Adopting the convention $\vn = (v_3, \ldots, v_{L})\in\RR_+^{L-2}$, while $\cn = (u_1, \ldots, u_K)\in\RR^K$ and  $\x = (x_1,\ldots, x_{K+L-2})\in\RR^{K+L-2}$, $\dom(\eps)$  is defined by
\begin{eqnarray}
\dom(\eps) \equiv \lbr \x\equiv(\cn,\vn) \in \RR_+^K\times\RR_+^{L-2} : \quad 
x_K\geq \eps; \;\; 
1\geq \sum_{j=3}^Lv_j; \;\; d(\vn) - \sum_{i=1}^Kix_i  \geq \eps   \rbr,
\end{eqnarray}
where $\RR_+ = \{x\in\RR : x\geq 0\}$ and as before we have $d(\vn) \equiv \sum_{j=3}^Lj v_j$.

\begin{propo}\label{lem:hyperasymptotic}
Fix $\eps\in(0,1)$ and $\dom = \dom(\eps)$. For $\x=(\cn,\vn) \in \dom$ and $\D\oz=(\D\vo,\D\t)\in\ints^{K+L-2}$ define transition kernel $W(\cdot|\cdot)$ as follows: $W(\D\oz|\x) = 0$ unless there exists a unique $\ell$ with $\D\ta_i = -\1_{i=\ell}$, in which case
\begin{eqnarray}
W(\D\oz | \x ) \equiv \frac{\ell v_\ell}{d(\vn)}
\binom{\ell -1}{q_0-1, q_1, \dots, q_{K}} \pp_0^{q_0-1}\pp_1^{q_1}\cdots\pp_K^{q_{K}}\label{eq:simplified}
\end{eqnarray}
with
\begin{eqnarray}
\pp_i = \frac{(i+1)x_{i+1}}{d(\vn)}\;\;\mbox{for}\;\;0\leq i \leq K-2;\quad
\pp_{K-1} = \frac{x_K\la^K}{(K-1)!e_K(\la) d(\vn)};\quad
\pp_K = \frac{x_K\la}{d(\vn)},
\end{eqnarray}
where $\la$ is the unique solution of 
\begin{eqnarray}\label{def:lam}
f(\la)\equiv\frac{\la e_{K-1}(\la)}{e_K(\la)} = \frac{d(\vn) - \sum_{i=1}^{K-1}ix_i}{x_K}
\end{eqnarray}
and where $q_i = - \sum_{j=i+1}^K \D z_j \geq \1_{i=0}$ for $i=0,\dots,K-1$ and $q_K = \ell + \sum_{i=1}^Ki \D z_i = \ell - \sum_{i=0}^{K-1}q_i \geq 0$.

Then there exists a positive constant $\sfk_2=\sfk_2(L,K,\eps)$ such that for any $\oz \in \ints^{K+L-2}\cap n\dom$ with $\om_i\geq 1$ for all $i$ we have
\[
\l| W_n(\cdot | \oz) -  W(\cdot | n^{-1}\oz) \r|_{\sTV} \leq \sfk_2 n^{-1}.
\]
\end{propo}
\begin{proof} 
We begin by defining
\begin{eqnarray}
\pi_1(\ozp,\oz) &\equiv& \frac{\binom{m}{z_1,\ldots,z_K}}{\binom{m}{z'_1,\ldots,z'_K}}
\frac{h(\ozp)}{h(\oz)} (\ta + 1)\ell ! \nonumber\\
&=&
n^{-\ell+1} (\ell-1)! \frac{\ell\, v_{\ell}}{d(\vn)} 
\frac{\prod_{i=0}^{K-1}(i!)^{-\D z_i}}{[d(\vn)]^{\ell-1}}
\frac{ \coeff[e_K(\sfx)^{z'_K}, \sfx^{\ovd'}] }
     { \coeff[e_K(\sfx)^{z_K} , \sfx^{\ovd}  ] } 
(1 + \gam^{(1)}_n), \label{eq:pi1}
\end{eqnarray}
where $|\gam^{(1)}_n| \leq c_1/n$ for some positive constant $c_1=c_1(\eps,L \,)$ and the equality follows from the estimate
\begin{eqnarray*}
\binom{d(\t)}{\ell}^{-1} = n^{-\ell+1}\frac{\ell}{d(\t)} \frac{(\ell-1)!}{[d(\vn)]^{\ell-1}}(1+\gam^{(1)}_n),
\end{eqnarray*}
itself a consequence of $\eps \leq d(\vn\,') < d(\vn) \leq L$ for $n^{-1}\oz \in \dom$ and $n^{-1}\ozp \in \dom$. We also define
\begin{eqnarray*}
\pi_2(Q,\ozp,\oz) \equiv \frac{\binom{m}{z'_1,\ldots,z'_K}}{\binom{m}{z_1,\ldots,z_K}}
\frac{N(\ozp|\oz)}{(\ta+1)\ell ! } =
\sum_{\cD} \binom{z_1-1}{q_{01}-1} 
\prod_{j=2}^{K} \binom{z_j}{\sum_{i=0}^{j}q_{ij}} \pi_3(Q),
\end{eqnarray*}
where for each $Q \in \cD$ 
\begin{eqnarray*}
\pi_3(Q) \equiv \prod_{j=2}^{K-1}\binom{\sum_{i=0}^{j-1}q_{ij}}{q_{0j},\dots,q_{(j-1)j}} 
\binom{\sum_{i=0}^{j-1}q_{ij}+q_{KK}}{q_{0K},\dots,q_{KK}} \pi(Q),
\end{eqnarray*}
and the equality follows by simple algebra.
Notice that $W_n(\D\oz|\oz) = \pi_1(\ozp, \oz) \pi_2(Q,\ozp,\oz)$ for $\oz$ and $\ozp$ in $n\dom$. We proceed to make a series of estimates to find the asymptotics of $\pi_1$, $\pi_2$, and $\pi_3$.

First, the condition $L \geq \ell \geq \ell - \ovell$ implies that 
$|\cD| < L^{K^2}$ and that $\pi_3(q)$ is uniformly bounded for $q \in \cD$ by some constant $c_2=c_2(L)$, since we always have $\pi(Q) \leq e^L$.
It follows that each term indexed by $Q$ in the sum over $\cD$ in $\pi_2$ is at most of the size
\begin{eqnarray}
c_2\, (nL)^{ \sum_{0\leq i<j\leq K} q_{ij} + q_K -1 }\label{eq:termbound}
\end{eqnarray}
and we notice that inequality $\sum_{0\leq i<j\leq K} (j-i)q_{ij} + q_{KK} \leq \ell$ implies that for all $q$ such that either $q_{ij}>0$ for some $i+1 <j$ or $\sum_{i=1}^Kq_{(i-1)i} + q_{KK} < \ell$ we have (\ref{eq:termbound}) bounded by $c_2 L^{L-2} n^{\ell-2}$. Consequently the sum over such $Q$'s only contributes at most $c_3 n^{\ell-2}$ for $c_3 = c_2 L^{K^2 + L-2}$.

It remains to consider $Q$ such that $q_{ij}=0$ for all $i+1 <j$ and $\sum_{i=1}^Kq_{(i-1)i} + q_{KK} = \ell$. In this case, we define $q_{i-1} \equiv q_{(i-1)i}$ for $1\leq i \leq K$ and $q_K=q_{KK}$. It is straightforward to verify that with this definition we have $q_i = - \sum_{j=i+1}^K \D z_j \geq 0$ for $i=0,\dots,K-1$, while additionally $q_0 \geq 1$, and 
$q_K = \ell + \sum_{i=1}^Ki \D z_i \geq 0$. In particular, $q$ is unique. Moreover, $\ovell = q_K + q_{K-1}$ and consequently $\pi(q) = \coeff[e_1(\sfx)^{\ovell},\sfx^{\ovell}] = 1$ and $\pi_3(q) = \binom{q_{K-1}+q_{K}}{q_{K}}$. All of this implies
\begin{eqnarray*}
\pi_2(q,\oz,\ozp) = \binom{z_1-1}{q_{0}-1} 
\prod_{j=2}^{K-1} \binom{z_j}{q_{j-1}} \binom{z_K}{q_{K-1}+q_K}\binom{q_{K-1}+q_{K}}{q_{K}} + \gam^{(2)}_n ,
\end{eqnarray*}
where $|\gam^{(2)}_n| \leq c_2 L^{L-2} n^{\ell-2}$. Rescaling by $n^{-\ell+1}$ and using $x_i = n^{-1}z_i$ we recover
\begin{eqnarray}
\frac{(\ell-1)! \pi_2(q,\oz,\ozp)}{n^{\ell-1}} 
%
&=& \binom{\ell-1}{q_0-1,q_1,\ldots,q_K}
x_1^{q_0-1} \prod_{j=2}^{K-1} x_j^{q_{j-1}} x_K^{q_{K-1}+q_K} + \gam^{(3)}_n , \label{eq:pi2}
\end{eqnarray}
where $|\gam^{(3)}_n| \leq (c_2 L^{L-2}+c_7) n^{-1}$ and we have used the inequalities
\begin{eqnarray*}
x_j^q - c_5n^{-1} \leq \frac{n^{-q} z_j!}{(z_j - q)!} \leq x_j^q
\qquad\mbox{and}\qquad
x_j^{q-1} - c_6n^{-1} \leq \frac{n^{-(q-1)} (z_j-1)!}{(z_j - q)!} \leq x_j^{q-1}
\end{eqnarray*}
for any $0\leq q \leq L\wedge z_j$ with positive constants $c_5 = c_5(L)$ and $c_6 = c_6(L)$.

Consequently, since $-\D z_i = q_{i-1}-q_i$ for $1\leq i \leq K-1$ implies
\[
\prod_{i=0}^{K-1}(i!)^{-\D z_i} = \prod_{i=0}^{K-1}(i!)^{q_{i-1}-q_i}=
\lsq(K-1)!\rsq^{-q_{K-1}} \prod_{i=0}^{K-1}i^{q_{i-1}},
\]
substituting this into (\ref{eq:pi1}) and combining with (\ref{eq:pi2}) yields the estimate
\begin{eqnarray}
W_n(\D\oz| \oz) = W(\D\oz|\oz) \frac{ \coeff[e_K(\sfx)^{z'_K}, \sfx^{\ovd'}] } { \coeff[e_K(\sfx)^{z_K} , \sfx^{\ovd}  ] } 
\frac{e_K(\la)^{q_{K-1}}}{\la^{K q_{K-1}+q_K}}
+ \;\gam^{(4)}_n\label{interestingline}
\end{eqnarray}
for some $|\gam^{(4)}_n| \leq c_8(\eps,L) n^{-1}$ and to obtain result (\ref{eq:simplified}) it only remains to show
\begin{eqnarray}
\l| \frac{ \coeff[e_K(\sfx)^{z'_K}, \sfx^{\ovd'}] } { \coeff[e_K(\sfx)^{z_K} , \sfx^{\ovd}  ] } 
\frac{e_K(\la)^{q_{K-1}}}{\la^{K q_{K-1}+q_K}} - 1 \r| \leq c_9n^{-1}, \label{eq:lastbit}
\end{eqnarray}
where the positive term $e_K(\la)^{q_{K-1}}\la^{-K q_{K-1}-q_K}$ does not depend on $n$.

To this end, note that for $\la>0$ and integers $t,s \geq 1$ we have
\begin{eqnarray}
p_\la(t,s) \equiv  \coeff[e_K(\sfx)^t, \sfx^s] \la^s e_K(\la)^{-t} = \P_\la \ll\sum_{i=1}^tN_i=s\rr,
\end{eqnarray}
where $\{N_i\}$ i.i.d. $\Poisson(\la)$ random variables conditioned on being greater or equal to $K$.
Hence, in view of $-\D z_K = z_K - z'_K = q_{K-1}$ and $\ovd-\ovd\,'= Kq_{K-1}+q_K$,
estimate (\ref{eq:lastbit}) is equivalent to
\begin{eqnarray}
\l| \frac{p_{\la}(z_K+\D z_K, \ovd - (Kq_{K-1} + q_K) ) } {p_{\la}(z_K, \ovd)  } 
 - 1 \r| \leq c_9n^{-1}. \label{eq:lastbitrestated}
\end{eqnarray}
This will follow from a local CLT for the sum, $S_k$ of i.i.d. variables $X_i = (N_i - \xi)/M_2(K,\la)$ for $\xi = \ovd/z_K$ similarly to \cite[Lemma 4.5]{DM08}. First, a straightforward computation reveals that the first two cumulants of $N_1$ are
\begin{eqnarray*}
M_1(K,\la) &=& \E_\la(N_1) = \frac{\la e_{K-1}(\la)}{e_K(\la)} \\
M_2(K,\la)^2 &=& \Var_\la(N_1) = 
\frac{\la [  \la e_{K-2}(\la)e_K(\la) + e_{K-1}(\la)e_K(\la)  - \la e_{K-1}(\la)^2 ] }{e_K(\la)^2}
\end{eqnarray*}
and similarly the higher normalized moments $M_q(K,\la) = \E_\la[N_1-M_1(K,\la)]^q/M_2(K,\la)^q$, $q>3$ can be computed. It is clear that each $M_q(K,\la)$, $q\geq 1$ is bounded away from zero and infinity for $\la$ bounded away from zero and infinity. Second, recall that by our assumption of $n^{-1}\oz \in \dom$ we have  $K + \eps/\rho \leq  \xi \leq L/\eps$ and this implies that the unique non-negative solution of $M_1(K,\la) = \xi$ (whose existence is shown as part of the proof of Lemma \ref{lem:smoothness}) is also bounded away from zero and inifinity. It follows that each moment $M_q(K,\la)$ is uniformly bounded on $\dom$, as well as that $p_{\la}(z_K, \ovd) =  \P(S_k = 0)$ and
$p_{\la}(z_K+\D z_K, \ovd - (Kq_{K-1} + q_K) ) = \P(S_{k'} = a)$ with 
$k = z_K$, $k'-k = \D z_K \in \{-(L-1),\ldots,0\}$ and $a = -  (Kq_{K-1} + q_K + \xi \D z_2)/M_q(K,\la)$. The last quantity, $a$ is uniformly bounded on $\dom$ and in the lattice of all possible values of $S_k$, which has span $b = M_2(K,\la)^{-1}$. At this point the proof is completed with exactly the same arguments as that of \cite[Lemma 4.5]{DM08}.
\end{proof}

\begin{lemma}\label{lem:smoothness}
Functions $\x \mapsto \pp_i$, $i=0,\ldots,K$ are three times differentiable for all $\x \in \dom$ with bounded continuous derivatives.
\end{lemma}
\begin{proof}
Fix $\eps > 0$ in the definition of $\dom$. We choose a $0 < \de < \eps$ and in the remainder of the proof consider $\dom^{\de}$, which will conclusively demonstrate differentiability at the boundary as the formulas for $\pp_i$ remain well-defined on $\dom^{\de}$ for each $\D\z$ (though this extension does not lead to an extended probability kernel, as for example $\pp_1 < 0$ for $\z$ with $x_1 < 0$).

It is enough to show that partial derivatives (of all orders) of $\x \mapsto \pp_i$ for $0\leq i \leq K$ exist in $\dom^{\de}$. This is clear for $i \leq K-2$, whereas for $i=K-1,K$ we notice that function $\la \mapsto \la^K/e_K(\la)$ is bounded and smooth on any compact subinterval of $(0,\infty)$, so that we only need to show that $\x \mapsto \la(\x)$, defined to be the solution of (\ref{def:lam}),
is bounded away from 0 and infinity (uniformly on $\dom^{\de}$) and infinitely differentiable. 

To this end, we first note that extending $f$ to the whole interval $[0,\infty)$ by setting $f(0)=K$ yields a monotone increasing, infinitely differentiable function from $[0,\infty)$ to $[K, \infty)$ and its first derivative is 
$f'(\la) = [(e_{K-1}(\la) + \la e_{K-2}(\la))e_K(\la) - \la e_{K-1}(\la)^2]/e_K(\la)^2$, which is bounded below by some positive constant (compare coefficients in front of powers of $\la$). As a consequence of the inverse mapping theorem, $f^{-1}$ is well-defined and infinitely differentiable on $[K,\infty)$, which further implies that $\la(\cdot)$ is well-defined, bounded and infinitely differentiable on $\dom^{\de}$ with bounded derivatives, because $x_K \geq \eps-\de>0$ there. This completes the proof.
\end{proof}

To determine the mean $\F$ and covariance matrix $\G$ of the increment with distribution $W(\cdot|\x)$, observe that the random variables $-(\D\ta_3, \ldots, \D\ta_L) \equiv -\D\t$ have a multinomial law with parameters $\{1; \ss_3, \ldots, \ss_L\}$ for $\ss_j \equiv jv_j/d(\vn)$ and conditionally on $\D\t$ we have $(q_0-1, q_1, \ldots, q_{K})$ distributed with a multinomial law with parameters $\{\ell-1; \pp_0, \ldots, \pp_{K}\}$, where $\ell$ equals the unique $j$ with $\D\ta_j = -1$. A simple computation then reveals that  $\F = (\Fx,\Fy)^\dag$ for
\begin{eqnarray}\label{def:extraF}
\Fx(\x) &=& \sfR'(1)\ll  -\frac{1}{\sfR'(1)} + \pp_1-\pp_0, \;\;  \pp_2-\pp_1, \ldots,\;\;  \pp_{K-1}-\pp_{K-2}, \;\; - \pp_{K-1}  \rr\\
\Fy(\x) &=&  \ll -\ss_3\,,\;\;\dots,\;\;\; -\ss_L  \rr
\end{eqnarray}
where $\sfR(\xi) \equiv \sum_{j=3}^L\ss_j\xi^{j-1}$, so that $\sfR'(1) = d(\vn)^{-1} \sum_{j=3}^Lj(j-1)v_j$, whereas $\G(\x) \equiv \Cov(\D\oz)$ with $\D\oz = (\D\vo,\D\t)$ is determined by ($1 \leq i \leq j \leq K$, $3\leq l \leq k \leq K$)
\begin{eqnarray} \label{def:extraG}\lbr
\begin{array}{rcl}
\Cov(\D \om_i,\D \om_i) &=& \sfR'(1)(\pp_i\1_{i\neq K} + \pp_{i-1})  
+\; [\sfR''(1) - (\sfR'(1))^2](\pp_i\1_{i\neq K} -\pp_{i-1})^2\; \\
\\
\Cov (\D \om_i, \D \om_j) &=& \sfR'(1) \pp_i\1_{i=j-1} +\; [\sfR''(1) - (\sfR'(1))^2](\pp_i -\pp_{i-1})(\pp_j\1_{j\neq K} -\pp_{j-1})\\
\\
\Cov (\D \om_i, \D\ta_l) &=& \ss_l \lsq l-1 - \sfR'(1) \rsq (\pp_i\1_{i\neq K} - \pp_{i-1})\;\\
\\
\Cov (\D\ta_l,\D\ta_k) &=& -\ss_l\ss_k + \ss_k\1_{l=k}\;\\
\end{array}\right.\end{eqnarray}

\subsection{Ensemble $\cG(n,m,\sfv_n)$}\label{sec:ensemble1}

In Section \ref{sec:hyperexact} we have only assumed that $G(0)$ is such that conditioned on $\{ G(0) \in \cG(\z) \}$, it is a uniformly random element of $\cG(\z)$. This condition allows for variety of initial ensembles. In this section we define a particular initial ensemble.

The inital ensemble $\cG \equiv \cG(n,m,\sfv_n)$ is characterized by non-negative integers  $n$, $m$, $\{\sfv_n(j)\}_{j=3}^L$ and $L\geq 3$. We require that $n = \sum_{j=3}^L\sfv_n(j)$ and that there exists a distribution vector $\vn_0 = (v_{03},\ldots,v_{0L})^\dag$ and a constant $c_0$ such that for each $j$ we have $|\sfv_n(j) - n v_j| \leq c_0$ uniformly in $n$. For simplicity, we also assume that $\sfv_n(j)=0$ for all $n$ if $v_j = 0$ (this can be easily lifted, but the notational burden involved will obscure the issue).

Let $\sfV(x) \equiv \sum_{j=3}^Lj v_{0j} x^{j-1}$ and note that the asymptotic average degree of the v-nodes $\mu \equiv \sfV(1)$ is independent of $n$.  An element in the ensemble is a graph
\begin{eqnarray}
G=(V_3,V_4,\ldots,V_L;C;E),
\end{eqnarray}
where $V_3\amalg\cdots\amalg V_L\equiv[n]$ is the set of v-nodes, with $|V_j| = \sfv_n(j)$, $C\equiv[m]$ is the set of c-nodes, and $E$ is an ordered list of edges
\begin{eqnarray*}
 E = [ (\al_1,a_1),(\al_1,a_2),\ldots,(\al_1,a_{l_1}); (\al_2,a_{l_1+1}),(\al_2,a_{l_1+2}),\ldots,(\al_2,a_{l_1+l_2}); \ldots;\\
(\al_{n},a_{l_1+\cdots + l_{n-1} + 1}),\ldots,(\al_{n},a_{l_1+\cdots + l_n})  ] 
\end{eqnarray*}
where a couple $(\alpha,a)$ appears before $(\beta,b)$ whenever $\alpha<\beta$ and each 
$\al \in V_j$ appears in \emph{exactly} $j$ pairs. The total number of graphs in this ensemble is 
\begin{eqnarray}
|\cG(n,m,\sfv_n)| = 
 \binom{n}{\sfv_n(3),\ldots,\sfv_n(L)} \coeff[(e^\sfx)^m,\sfx^n]\; \ll\sfh_n\rr!
\end{eqnarray}
where $\sfh_n \equiv \sum_{j=3}^Lj\sfv_n(j)$.

To sample from this distribution, first partition $[n]$ into disjoint sets $\{V_j\}_{j=3}^L$ and attribute $j$ sockets to each v-node in the set $V_j$. Second, attribute $k_a$ sockets to each c-node $a$, where $k_a$'s are mutually independent $\Poisson(\mu/\rho)$ random variables, conditioned on the event $\{ \sum_{a\in C}k_a = n\mu \}$. Finally, connect the v-node sockets to the c-node sockets according to an indepedently and uniformly chosen permutation of $\{1,\ldots,n\mu\}$. This sampling procedure is used to establish the approximate mean and covariance matrix for the initial state $\oz = (\vo,\t)$ of our Markov chain. Clearly we only need to deal with $\vo(0) = \orn$, as $\t(0)$ is non-random with entires $\ta_j(0) = \sfv_n(j)$ for $3\leq j \leq L$.

\begin{lemma} \label{lem:hyperinitial}
Let $\orn = (\om_1,\om_2)$ denote the number of c-nodes of respectively degree 1 and degree strictly greater than 1 in a graph $G$ chosen uniformly at random from the non-empty ensemble $\cG(n,\fl n\rho  \fr,\sfv_n)$. Then, for any $\eps>0$, there exist finite, positive constants $\ka_i$ such that for all (allowed) $n$, all $r$, $N$, and $\rho \in [\eps,1/\eps]$
\begin{eqnarray}
|| \E\orn - n \cn_\rho || &\leq& \ka_0  \label{eq:follow1}\\ 
\P \{ \l| \orn - \E\orn \,\r| \geq r\} &\leq& \ka_1 e^{-r^2/(\ka_2 n)}  \label{eq:follow2} \\
\sup_{\U \in {\cal M}_{N\times 2}} \sup_{\x \in \RR^2} \ln 
\P\{\U \orn \leq \x \} - 
 \P\{n^{1/2} \U \vze_\rho \leq \x \}  \rn &\leq& \ka_3 n^{-1/2}. \label{eq:stronger}
\end{eqnarray}
%
%
where $\vze$ is a Gaussian vector in $\RR^2$ with mean $\cn_\rho$ and covariance $\Qc_\rho$ where 
\begin{eqnarray*}
\cn_\rho &=& \ll \mu e^{-\gam}, \quad \rho(1-e^{-\gam}) - \mu e^{-\gam} \rr^\dag\\
\relax [\Qc_\rho]_{11} &=& \mu e^{-2 \gam} (e^\gam - 1 + \gam - \gam^2)  \\
\relax [\Qc_\rho]_{12} &=& -\mu e^{-2 \gam} (e^\gam - 1 - \gam^2) \\
\relax [\Qc_\rho]_{22} &=&  \rho e^{-2 \gam} [ (e^\gam - 1) + \gam(e^\gam - 2) - \gam^2(1+ \gam)]
\end{eqnarray*}
where $\gam \equiv \sfV(1)/\rho$ is the average degree of a c-node and $\mu \equiv \sfV(1)$ is the average degree of a v-node. 
%
\end{lemma}

\begin{proof} 
Given $n$ and $m = \floor{n\rho}$, by construction the initial distribution of the v-nodes only affects the distribution of the c-nodes through the total number, $\sfh_n$, of the v-node sockets, which has to match the number of the c-node sockets. Moreover, we have $|\sfh_n - n\mu| \leq c_0$. Hence the analysis reduces to that for the regular hypergraph ensembles, as in the proof of \cite[Lemma 4.4]{DM08}. We follow that proof (subtituting each occurence of $nl$ with that of $\sfh_n$ and using the approximation by $n\mu$), thus demonstrating (\ref{eq:follow1}) and (\ref{eq:follow2}), which correspond to \cite[(4.18) and (4.19)]{DM08} respectively, until the end of the paragraph containing the estimate \cite[(4.25)]{DM08}. The conclusion of that paragraph is that for each $0<\eps'<1$ there is a positive constant $c_0=c_0(\eps')$ such that for each $\gam\in[\eps',1/\eps']$ we have
\begin{eqnarray}
\sum_{\z \in \ints^2}\ln \P(\z) - \sfG_2(\z | n'\cn_\rho; n'\Qc_\rho) \rn \leq c_0 n^{-1/2},\label{eq:fromDM08}
\end{eqnarray}
where $n'=m/\rho$, $\P$ is the probability distribution of the vector $\zrn$, and $\sfG_2(\cdot|\x;\A)$ is the Gaussian density in $\RR^2$ with mean $\x$ and covariance matrix $\A$.

At this point we diverge from the proof of \cite[Lemma 4.4]{DM08} to prove the stronger conclusion (\ref{eq:stronger}). Clearly (\ref{eq:fromDM08}) implies (from now on supressing dependence on $\rho$)
\begin{eqnarray}
\sup_{\U \in {\cal M}_{N\times 2}} \sup_{\x \in \RR^2} \ln 
\P\{\U \z_n \leq \x \} - 
\sum_{\U \z \leq \x }  \sfG_2(\z | n'\cn; n'\Qc)  \rn &\leq& c_0 n^{-1/2}. \label{eq:n'ton}
\end{eqnarray}
Let $\xx$ and $\yy$ be Gaussian vectors with mean and covariance matrix pairs 
$( n'\cn/\sqrt{n}, n'\Qc/n)$ and $( n\cn/\sqrt{n}, n\Qc/n)$ and notice that, since $\Qc$ is positive definite, the lowest eigenvalues of both the covariance matrices are bounded from below by a positive constant uniformly in $n$ and $\gam$. Then, since $|n - n'| < \rho \leq 1/\eps$, an application of Lemma \ref{lem:gaussianbound} to the family of convex sets 
$\cC' \equiv \{n^{-1/2}C : C \in \cC, n > 0 \}$ gives a positive constant $c_1$ such that for all $n$ and $\gam$
\begin{eqnarray}
\sup_{C \in \cC} \ln 
\Pn\{\sqrt{n}\xx \in C\} - 
\Pn\{\sqrt{n}\yy \in C \}  \rn &\leq& c_1 n^{-1/2},
\end{eqnarray}
implying that we can exchange $n'$ for $n$ in (\ref{eq:n'ton}). Now, the sets $C_{\U, \x, n} \equiv \{\z\in\RR^2 : \U(\z - n\cn) \leq \x\}$ are certainly convex and Borel-measurable, so the proof will be complete upon demonstrating that for some $c_2>0$
\begin{eqnarray}
\sup_{C \in \cC} \ln 
\sum_{\z \in C }  \sfG_2(\z\,|\,\vec{0}; n\Qc) - 
\int_{\z \in C} \sfG_2(\z\,|\,\vec{0}; n\Qc) \did\z \rn &\leq& c_2 n^{-1/2}.
\end{eqnarray}
holds uniformly in $n$ and $\gam$. This is a consequence of the Euler-MacLaurin sum formula, as stated, for example, in \cite[Theorem A.4.3]{BR76}, where we take $f$ to be the Schwarz function $\sfG_2(\cdot\,|\, \vec{0}; \Qc)$ and where the first correction to the Gaussian distribution in $\Lambda_1(\x)$ of \cite[(A.4.20)]{BR76} is at most $c_3 n^{-1/2}$.
\end{proof}

Notice that for $K=2$ the ODE system $\did \y/\did \th = \F(\y)$ where $\y = (\cn,\vn)$ becomes
\begin{equation}\lbr \begin{array}{rcl}
\dfrac{\did u_1}{\did \th} &=& -1 + 
\dfrac{\sfR'(1)}{d(\vn)} \lsq \dfrac{\la^2}{e_2(\la)} u_2 - u_1\rsq \\
\\
\dfrac{\did u_2}{\did \th} &=&  - \dfrac{\sfR'(1)}{d(\vn)} \dfrac{\la^2}{e_2(\la)} u_2\\
\\
\dfrac{\did v_j}{\did \th} &=& -\dfrac{j v_j}{d(\vn)}\qquad\mbox{for } j=3,\ldots,L
\end{array}\right.\label{firstsystem}
\end{equation}
with the initial condition $\y(0) = \y_\rho = (\cn_\rho,\vn_0)$. (For $\rho=\rhoc$, this is exactly the system defined by (\ref{def:F}) and (\ref{def:y0}).) Its solution is given in Lemma \ref{lem:hypersolution}.

\subsection{Proof of Theorem \ref{thm:hypergraph}}\label{1stproof}
Fix $0 < \eta_0 <\rhoc$. Let $\tast$ denote the first time at which $z_1(\ta) \leq 0$. Since the 2-core is the stopping set including the maximal number of v-nodes, Lemma \ref{lem:smallcores} implies the existence of constants $\ka_1,\ka_2$ such that for all $\rho \geq \eta_0 n$ the probability that a random hypergraph from $\cG(n,\floor{n\rho},\sfv_n)$ has a non-empty core of size less than $\ka_1n$ is at most $\ka_2 n^{1-l/2}$. Thus setting $\eta_1 = (\ka_1 \eta_0 \wedge \eta_0)/2$ and 
$P(n,\rho) \equiv \Pn\{ \min_{0\leq \tau\leq n(1-\eta_1)}z_1(\ta) \leq 0\}$ we have
\[
| \Pn \{ \tast \leq n-1 \} - P(n,\rho) | \leq \ka_2 n^{-1/2},
\]
and fixing in the remainder of the proof the domain $\dom = \dom(\eta_1)$, the statement of the theorem will follow from Theorem \ref{thm:markovexit} once we show that its hypothesis holds.

Let $\vec{n} \equiv (n_1, n_2, \ldots, n_{L})$, where $n_1 =n_2 =0$ and $n_j = \1_{v_j = 0}$ for $3\leq j \leq L$, and let $d = ||\vec{n}|| \in \ints$. Then $\amb \equiv \{\x \in \RR^L: \vec{n}\cdot\x = 0\} \cong \RR^d$ and the image of $\dom\cap\amb$ under this last vector space homomorphism is a non-degenerate compact convex set $\sfD \subset \RR^d$. 
To demonstrate the hypothesis, observe that by Lemmas \ref{lem:hyperexact}, \ref{lem:hyperasymptotic}, and \ref{lem:smoothness} our original chain has the countable statespace $\ints^{d}\cap n\sfD$, is bounded and density dependent, with a sufficiently smooth kernel. Moreover, Lemma \ref{lem:hyperinitial} states that the parametrization of the intitial distribution satisfies the relevant conditions for Theorem \ref{thm:markovexit} (where the $\rho$ in Theorem \ref{thm:markovexit} is $(\rho - \rhoc)$ for the $\rho$ in the Lemma  \ref{lem:hyperinitial}). It remains to be shown that the critical solution of the fluid limit ODE has the required properties and that the relevant matrix $\bS$ is positive definite.

We start with the former. The solution of \ref{firstsystem} in Lemma \ref{lem:hypersolution} reveals that $\Theta \equiv \{\th\in (0,1) : x_1(\th, \rhoc)=0 \} = \{\th\in(0,1) : \sfV(\ze(\th))/\rhoc = -\log(1-\ze(\th)) \}$. Since $g(\ze) \equiv -\log(1-\ze)$ is analytic on $[0,1)$ and $\lim_{\ze \uparrow 1}g(\ze) = \infty$, so it can only  coincide with the polynomial $\sfV$ at finitely many points in $(0,1)$. Hence $\Theta = \{\thic: i=1,\ldots,N\}$, as defined in the statement of the result. By the first paragraph of the proof we must also have $\th_c^N < 1-\eta$ and we can now set $\ovth = [\th_c^N + (1-\eta_1) ]/2$. Differentiability of $(\ze,\rho) \mapsto \sfV(\ze)/\rho + \log(1-\ze)$ on $(0,1)\times[\eps,1/\eps]$ and the inverse function theorem imply that $\sfV'(\ze) = g'(\ze) = (1-\ze)^{-1}$ for each $\ze \in Z$. Since we assumed that $\sfV''(\ze)/\rhoc < (1-\ze)^{-2}$ for each $\ze \in Z$, inspection of the solution $\y = (\cn,\vn)$ in Lemma \ref{lem:hypersolution} and elementary calculus allow us to conclude that the critical solution remains in $\sfD^\circ$ until time $\ovth$  ($u_2(\th)>0$ for $\th\in [0,\ovth]$, same for $v_j$ with $v_{0,j}>0$) except at times $\th\in \Theta$, when it is tangent to the plane $x_1\equiv 0$ (and no other face of $\sfD$), moreover satisfying the tangency conditions in (\ref{tangencyconditions}) with $\sfd_i = 0$ and $\sfmr_i = \vec{\sfe}_1$.

The proof will be complete once it is demonstrated that the p.s.d. symmetric matrix with entries ($1\leq i \leq j \leq N$)
\[
[\bS]_{ij} = \lsq\Q(\thic) (\B_{\thic}(\thjc))^\dag \rsq_{11},
\]
is actually positive definite. Here $\Q$ and $\B$ are defined as in Theorem \ref{thm:markovexit}, where we take all the chain specific data ($\F$, $\G$, $\y_0$ and $\Q_0$) like in Theorem \ref{thm:hypergraph}. Notice that $\bS$ is the covariance matrix of the centered Gaussian vector $\bY$ in $\RR^N$ with entries $[\bY]_j = \sum_{i=1}^j [\B(i,j)\bg_i]_1$, where $\B(i,j) = \B_{\thic}(\thjc)$ and $\{\bg_j\}$ are independent centered Gaussian vectors in $\RR^{K+L-2}$ with covariance matrices 
\[
\bV_j = \int_{\thic}^{\thjc} \B_{\ze}(\thjc) \G(\y(\ze)) (\B_{\ze}(\thjc))^\dag \did \ze
+ \Q(0)\1_{j=1}.
\]
It is now enough to show that $[\B(j,j)\bV_j\B(j,j)^\dag]_{11}=[\bV_j]_{11}>0$ for every $j=1,\ldots,N$. Since $\det\B_\ze(\th)>0$ and $\G$ is p.s.d., this will follow once we show that
$[\B(\eps)\G(\x)\B(\eps)^\dag]_{11}>0$ for all $0 \leq \eps < \eps_0$ and $\x\in\sfD$, where $\B(\eps) \equiv \B_{\thjc-\eps}(\thjc)$. This last claim is a consequence of $G_{11}(\x)>0$ for $\x\in\sfD$, $\B(0)=\ind$ and the continuity of $\B(\cdot)$.
\endproof

\begin{remark}\label{rem:final}
We briefly explain, using the example of ensemble $\cGt \equiv \cG(n,\rho,\sfu_n,\sfv_n)$, how to analyze other ensembles within this framework. Ensemble $\cGt$ is defined similarly to $\cG(n,\floor{\rho n},\sfv_n)$, except that the degrees of hypervertices are given by a distribution $\sfu_n$. There are exactly $\sfu_n(i)$ hypervertices with degree $i$ and we assume $\floor{\rho n} = \sum_{i=1}^{K-1}\sfu_n(i)$, $\sum_{i=1}^{K-1} i \sfu_n(i) = \sum_{j=3}^L j \sfv_n(j)$ and $|\sfu_n(j) - n \urj| \leq c_0$ uniformly in $n$ for some strictly positive vector $\cn_{\rho} =(u_{\rho,1,}\ldots,u_{\rho,K})^\dag$, such that the map $\rho \mapsto \cn_\rho$ is twice continuously differentiable.

To establish the FSS for the phase transition in this ensemble, first, notice that the proof of Proposition \ref{lem:hyperasymptotic} also establishes the simpler (do not require the CLT arguments) claim that the kernel $W_n$ is density dependent on
\begin{eqnarray}
\widetilde{\dom}(\eps) \equiv \lbr \x\equiv (\cn,\vn) \in \RR_{+}^K\times\RR_{+}^{L-2} : \qquad 
  1\geq \sum_{j=3}^Lv_j \geq \eps; \quad d(\vn) = \sum_{i=1}^Kix_i  \rbr
\end{eqnarray}
considered as a subset of the unique $(K+L-3)$-dimensional affine hyperplane that contains it (it follows from the definition of $W_n$ that a chain with this kernel remains in $n\domt$ if started there). The asymptotic kernel in this case, $\Wt(\cdot|\cdot)$ is defined similarly to $W(\cdot|\cdot)$, where we substitute $\domt$ for $\dom$, take $\pp_{K-1} = K x_K/d(\vn)$ and $q_K = \pp_k= \la = 0$, impose $\ell = - \sum_{i=1}^{K}i \D z_i =  \sum_{i=0}^{K-1}q_i$, and interpret $\pp_K^{q_K}=0^0$ as 1. Second, the associated differential equations are given in terms of $\F$ and $\G$ computed in (\ref{def:extraF}) and (\ref{def:extraG}), respectively, and the initial condition is now $\y_\rho=(\cn_\rho,\vn_0)$ and $\Q_\rho = 0$. The solution of such a system is provided in \cite[Appendix B]{LMSS01}. The absence of small cores in this ensemble can be established with a variation of the proof of Lemma \ref{lem:smallcores}. The analog of Theorem \ref{thm:hypergraph} for this ensemble now follows by essentially the same proof as above.
\end{remark}

\section{Acknowledgements}
I am very grateful to Amir Dembo for suggesting the problem and many helpful discussions. Special thanks to Andrea Montanari for pointers to the coding and statistical physics literature.
%
%

\appendix
\section{Appendix:  Proof of  Lemma \ref{lem:meanvar}}\label{app:A}

The matrix $\A(\y(\ta/n,\rho))$ and so also matrices $\tA_\ta$ are uniformly (in $\ta$ and $\rho$) bounded in the operator norm, which implies uniform (in $\si$, $\ta$ and $\rho$) boundedness of $||\tB_{\si}^{\ta}||$ and for $n > n_0$, also the existence of the inverses $(\tB_{\si}^{\ta})^{-1}$ and a uniform bound (in $\si$, $\ta$ and $\rho$) on their norm.

Moving on to demonstrate (\ref{eq:ystar}), let $D_n(\ta) \equiv n^{-1}\yst(\ta) -  \y(\ta/n, \rho) $, so that $D_n(0)=0$ and $D_n(\ta+1) = \tB_{\ta}^{\ta} D_n(\ta) + \vxi_n(\ta)$, where
\[
\vxi_n(\ta) = \int_{\frac{\ta}n }^{\frac{\ta+1}n} \lsq\F(\ytnrho) -\F(\y(\th,\rho))\rsq \did\th.
\]
The Lipschitz continuity of $\F(\cdot)$ and $\y(\cdot,\cdot)$ together with the supremum bound on the integral give the uniform estimate
\[
\l| \vxi_n(\ta) \r| \leq c_0 n^{-2}.
\]
With $D_n(\ta) = \sum_{\si=0}^{\ta-1}\tB_{\si+1}^{\ta} \vxi_n(\ta)$ for $\ta \leq n \ovth$\; and $\tB_{\si}^{\ta}$ uniformly bounded, as explained in the first paragraph, we can deduce (\ref{eq:ystar}).

To see that (\ref{eq:tB}) holds, define $D'_n(\si,\ta) \equiv || \tB_\si^{\ta-1} - \B_{\si/n}(\ta/n)||$, so that in view of $\tB_{\si}^{\si-1}=\B_{\si/n}(\si/n)=\ind$ we have $D'_n(\si,\si)=0$, while for $\ta \geq \si$
\begin{eqnarray}
D'_n(\si,\ta+1) &\leq& D'_n(\si,\ta) + n^{-1}\l| \tA_\ta\tB_\si^{\ta-1} + 
n \lsq \B_{\si/n}(\tfrac{\ta+1}n) - \B_{\si/n}(\tfrac{\ta}n)\rsq \r| \nonumber\\
&\leq& D'_n(\si,\ta) + n^{-1}\l| n \int_{\frac{\ta}{n}}^{\frac{\ta+1}{n}}
\lsq \tA_\ta\tB_\si^{\ta-1} + \A(\y(\th,\rho))\B_{\si/n}(\th) \rsq \did\th \r|  \nonumber\\
&\leq& D'_n(\si,\ta) + n^{-1}\sup_{\th \in [\frac{\ta}n, \frac{\ta+1}n] }
\l| \tA_\ta\tB_\si^{\ta-1} + \A(\y(\th,\rho))\B_{\si/n}(\th)\r|. \label{eq:D'}
\end{eqnarray}
Next, recall that $\tA_\ta = \A(\y(\ta/n,\rho))\1_{\ta \in I_n}$, so that boundedness (of $\A$) and Lipschitz continuity (of $\y$ and $\A$) give
\begin{eqnarray}
\sup_{\th \in [\frac{\ta}n, \frac{\ta+1}n] }\l| \tA_\ta - \A(\y(\th,\rho)) \r| 
\leq c_1 n^{-1}\1_{\ta \in I_n} + c_2\1_{\ta \in J_n}.\label{eq:supa}
\end{eqnarray}
From the ODE (\ref{ode:b}) and that same boundedness of $\A$ it is apparent that
\begin{eqnarray}
\l| \B_\ze(\th) - \B_{\ze'}(\th')\r| \leq c_3 ( |\th-\th'| + |\ze-\ze'| ).\label{eq:supb}
\end{eqnarray}
Uniform boundedness of all the matrices involved and the inequalities (\ref{eq:supa})
and (\ref{eq:supb}) imply that for $\th \in [\ta/n, (\ta+1)/n]$ we have
\begin{eqnarray*}
\l| \tA_\ta\tB_\si^{\ta-1} + \A(\y(\th,\rho))\B_{\si/n}(\th)\r| 
&\leq& c_4 D'_n(\si,\ta) + c_5n^{-1} + c_6 \1_{\ta \in J_n}
\end{eqnarray*}
which combined with the bound in (\ref{eq:D'}) results in the estimate 
\[
D'_n(\si,\ta+1) \leq (1 + c_4n^{-1}) D'_n(\si,\ta) +  c_5n^{-2} + c_6n^{-1} \1_{\ta \in J_n}
\]
and hence
\[
\max_{0 \leq \si \leq \ta \leq n\ovth} D'_n(\si,\ta) \leq  c_7 (1 + |J_n|) n^{-1}.
\]

Next, observe that for any $d\times d$ matrices $\V,\U$ and $\V',\U'$ we have
\begin{eqnarray*}
||\V\U\V^\dag - \V'\U'(\V')^\dag|| &\leq& ||\V-\V'|| ||\U\V^\dag|| +
||\V'\U|| ||\V-\V'||  + ||\V'|| ||\U-\U'|| ||\V'||.
\end{eqnarray*}
Applying this estimate (first with $\U = \U' = \Q(0,\rho)$, $\V = \B_0(\ta/n)$ and 
$\V' = \tB_0^{\ta}$, and then with $\U = \G(\y(\ze,\rho))$, $\U' = \G(\ysnrho)$, 
$\V = \B_\ze(\ta/n)$ and $\V' = \tB_{\ceil{n\ze}}^{\ta}$) and subsequently 
using the bound (\ref{eq:tB}), uniform boundedness of all the matrices involved and Lipschitz continuity of $\y(\cdot,\cdot)$ and $\G(\cdot)$ gives
\begin{eqnarray*}
|| \B_0(\ta/n)\, \Q_\rho ( \B_0(\ta/n) )^\dag - \tB_0^\ta\, \Q_\rho ( \tB_0^\ta )^\dag ||
&\leq&  c_8 (1 + |J_n|)n^{-1} \\
|| \B_{\ze}(\ta/n) \,\G(\y(\zeta,\rho)) ( \B_{\ze}(\ta/n) )^\dag -
\tB_{\si+1}^\ta \, \G(\ysnrho) ( \tB_{\si+1}^\ta )^\dag|| &\leq&  c_8 (1 + |J_n|)n^{-1}
\end{eqnarray*}
uniformly in $\ta$ and $\ze \leq \ta/n$. Finally, noticing that the sum over $\si\in I_n\cap [0,\ta]$ in the definition (\ref{def:var}) differs from the analogous sum over $\si\in[0,\ta]$ by at most $\ceil{|J_n|}$ uniformly bounded terms and comparing that last sum with the integral in
(\ref{def:Q}) gives the desired result.
\endproof
\section{Appendix: ODE solutions and absence of small cores}\label{app:B}
Here we provide solutions to the system of differential equations (\ref{firstsystem}) that arose from our analysis of the 2-core problem in hypergraphs. This is a special case of the result obtained in \cite{LMSS01}.

\begin{lemma}\label{lem:hypersolution}
There exists a unique solution $(\th,\rho) \mapsto \y(\th,\rho) = (\cn(\th,\rho),\vn(\th))$ in $C^3([0,\ovth]\times[-\mala,\mala],\RR^L)$ to the ordinary differential system (\ref{firstsystem}) with the initial condition $y_\rho = (\cn_\rho,\vn_0)$ and it satisfies the relation
\begin{eqnarray}
u_1(\th,\rho) &=& h(\th) \lsq\,\zeta(\th) - 1 + e^{-h(\th)/\rho} \,\rsq \label{eq:odesoln1}\\
u_2(\th,\rho) &=& \rho\, e^{-h(\th)/\rho}\,e_2\ll h(\th) / \rho \rr \label{eq:odesoln2}
\end{eqnarray}
where $h(\th) \equiv \sfV(\zeta(\th)) $ and
$\zeta(\th) \equiv \exp\{ - \int_0^\th d(\vn(s))^{-1} \did s \}$.
\end{lemma}
\begin{proof}
The existence and smoothness property of the solution comes from the corresponding properties of $\F$, which are a consequence of Lemma \ref{lem:smoothness}. In order to prove that (\ref{eq:odesoln1}) and (\ref{eq:odesoln2}) indeed give a solution, we start by making the substitution $x_i(\th,\rho) = x_i(\ze,\rho)$ for the monotonically decreasing continuously differentiable function $\ze:[0,1] \rightarrow [0,1]$ defined in the statement of the lemma. Notice that $\ze(0) = 1$ and $\lim_{\th \uparrow 1}\ze(\th) = 0$, since $\did d(\vn(s))/\did s < -3$ and $\lim_{s \uparrow 1} d(\vn(s)) = 0$, so that $\ze$ is a bijection. 

Next, substituting 
\[ 
x_1(\ze,\rho) = \sfV(\ze) \lsq\,\ze - 1 + \exp(-\sfV(\ze)/\rho) \,\rsq\qquad\mbox{and}\qquad x_2(\ze,\rho)=\rho\, e^{\sfV(\ze)}\;e_2\ll \sfV(\ze) / \rho \rr 
\]
for respectively $x_1$ and $x_2$ in (\ref{def:lam}) with $K=2$, we can verify that $y_j(\ze) = y_j(\ze=1) \ze^j$, which implies
\[
d(\y(\ze)) = \sum_{j=3}^Lj y_j(\ze) = \sfV(\ze)\ze \qquad\mbox{and}\qquad \la(\x(\ze), \y(\ze)) = \sfV(\ze)/\rho.
\]
Moreover, $\did \ze/\did \th = -1/V(\ze)$ and $\sfR'(1) = \sfV'(\ze)/\sfV(\ze)$, so substituting for $\la$, $x_1$ and $x_2$ in the ODE system w.r.t. $\th$ reduces it to the following systems w.r.t $\ze$
\begin{eqnarray}
\dfrac{\did x_1}{\did \ze}(\ze) &=& -\sfV(\ze) + \sfV'(\ze)e^{-\sfV(\ze)/\rho} \,\sfV(\ze)/\rho 
- \sfV'(\ze)  \lsq \ze - 1 + e^{-\sfV(\ze)/\rho} \rsq \\
\dfrac{\did x_2}{\did \ze}(\ze) &=&  - \sfV'(\ze)e^{-\sfV(\ze)/\rho} \,\sfV(\ze)/\rho
\end{eqnarray}
Simple differentiation now shows that the functions $x_1(\ze,\rho)$ and $x_2(\ze,\rho)$ indeed satisfy this system, together with the initial condition. 
\end{proof}
 
From our definition of the domain $\sfD$ it is clear that the Markov exit problem yields information about the existence of large 2-cores, that is 2-cores of size at least $\eta n$ for some positive $\eta$. In the language of coding theory, this means that it detects large decoding failures (lying in the so-called \emph{waterfall} region). In principle, there could exist cores of smaller than linear size (corresponding to decoding failures in the \emph{error floor} region). It is a consequence of the results in  \cite{OVZ05} that in  ensembles with minimal v-node degree greater than 2 the probability of such occurence tends to zero as $n$ increases. 

Here we prove an explicit bound on this small probability for our ensemble $\cG(n,m,\sfv_n)$, which is accomplished by estimating the probability that the hypergraph contains a small stopping set. A \emph{stopping set} is defined to be a subset of the v-nodes such that the restriction of the hypergraph to this subset has no c-nodes of degree 1. Thus the 2-core is a stopping set including the maximal number of v-nodes.

\begin{lemma}\label{lem:smallcores}
Consider the ensemble
$\cG(n,\floor{n\rho},\sfv_n)$ and let $l \equiv \min\{i: v_{0,i} > 0 \}$. Then for any $\eps>0$ there exists finite positive constants $\ka_1(\eps)$ and $\ka_2(\eps)$ such that for any $\rho \geq \eps$ the probability that a random hypergraph from the ensemble $\cG(n,\floor{n\rho},\sfv_n)$ has a stopping set consisting of fewer than $\ka_1 n$ v-nodes is at most $\ka_2 n^{1-l/2}$.
\end{lemma}
\begin{proof}
Write $m = \floor{n\rho}$. For a positive $\ka$ to be chosen, we will bound the expectation of the number $S$ of stopping sets consisting of fewer than $\ka m$ v-nodes. Indeed 
\begin{eqnarray}
\En S = \sum_{v=1}^{\ka m} \sum_{d=l v}^{L v} \sum_{r=1}^{\floor{d/2}} P_{v,e,r},
\end{eqnarray}
where $P_{v,d,r}$ is the expected number of stopping sets consisting of exactly $v$ v-nodes, with total degree $d$ and connected to exactly $r$ c-nodes.
We have the explicit representation
\begin{eqnarray}\label{1stensemblerep}
P_{v,d,r} = \coeff \lsq \prod_{j=l}^L (1+\sfy \sfz^j)^{\sfv_n(j)}, \sfy^v \sfz^d\rsq 
\frac{1}{m^d} \binom{m}{r} \coeff\lsq (e^\sfx -1 - \sfx)^r, \sfx^d\rsq d!,
\end{eqnarray}
where the first factor is the number of choices of $v$ v-nodes so that their total degree is $d$, $m^d$ counts the total number of ways of connecting these $v$ v-nodes in the graph ensemble, and $ \binom{m}{r} \coeff\lsq (e^\sfx -1 - \sfx)^r, \sfx^d\rsq d!$ counts the number of ways of connecting $r$ c-nodes to the $v$ v-nodes resulting in a stopping set. Since 
\begin{eqnarray}\label{coeffest}
\coeff \lsq \prod_{j=l}^L (1+\sfy \sfz^j)^{\sfv_n(j)}, \sfy^v \sfz^d\rsq  \leq \binom{n}{v}\qquad \mbox{and}\qquad
\coeff\lsq (e^\sfx -1 - \sfx)^r, \sfx^d\rsq \leq 1
\end{eqnarray}
we have, for $m \geq \eps n$, $r \leq \floor{d/2}$, $lv \leq d \leq Lv$ and $w \equiv d - l v$ the bound
\begin{eqnarray}
P_{v,d,r}  \leq \binom{n}{v}\binom{m}{r} \frac{d!}{m^d} 
\leq \ll\frac{e n}{v}\rr^v  \frac{m^{\floor{d/2}}}{\floor{d/2}!} \frac{d!}{m^d}
\leq \ll\frac{e n}{v}\rr^v  \ll \frac{d}{m} \rr^{lv/2 + w/2} 
\leq \lsq c_1 \ll \frac{v}{m}\rr^{l/2-1} \rsq^v \lsq \frac{Lv}{m} \rsq^{w/2}.\label{pvdrest}
\end{eqnarray}
We now fix $\ka >0$ small enough so that $(c_1 \ka^{l/2-1})\vee(\sqrt{L\ka}) \leq \tfrac12$, which allows us to obtain
\begin{eqnarray*}
\En S &\leq& \sum_{v=1}^{\ka m} \sum_{w=0}^{(L-l) v} \sum_{r=1}^{\floor{d/2}}
c_1 \ll \frac{v}{m}\rr^{l/2-1}  \lsq c_1 \ka ^{l/2-1} \rsq^{v-1} \lsq L\ka \rsq^{w/2} \\
&\leq& c_1 m^{1-l/2} \sum_{v=1}^{\infty} 
v^{l/2-1}  \lsq \frac12 \rsq^{v-1} \sum_{w=0}^{\infty} \floor{d/2} \lsq \frac12\rsq^{w} \\
&\leq& c_1 m^{1-l/2} \sum_{v=1}^{\infty} 
v^{l/2}  \lsq \frac12 \rsq^{v-1}
\end{eqnarray*}
and by comparison with the geometric series the last sum converges to a finite positive constant. Thus the thesis of the lemma follows with $\ka_1 = \ka \eps$ and $\ka_2 = c_1 \eps^{1-l/2} \sum_{v=1}^{\infty} v^{l/2} 2^{1-v}$.

\end{proof}


\bibliography{mythesis}

\end{document}